\documentclass[12pt,a4paper]{amsart}
\usepackage[latin1]{inputenc}   
\usepackage{amsmath}
\usepackage{amsfonts}
\usepackage{amssymb}
\usepackage{esint}
\usepackage{graphicx}
\usepackage{amsthm}
\usepackage{enumerate}
\usepackage{verbatim}
\usepackage{hyperref}
\usepackage[english]{babel}
\usepackage{bbm}
\usepackage{mathrsfs}
\usepackage{relsize}
\usepackage{textcomp}
\usepackage{xcolor}

\def \N{\mathbb{N}}
\def \R{\mathbb{R}}
\def \E{\mathbb{E}}

\theoremstyle{plain} 
\newtheorem{thm}{Theorem}[section] 
 
\newtheorem{lem}[thm]{Lemma} 
\newtheorem{prop}[thm]{Proposition} 
\newtheorem{hp}[thm]{Hypotheses} 
\newtheorem{defn}[thm]{Definition}

\theoremstyle{definition} 
\newtheorem{rem}[thm]{Remark}
\numberwithin{equation}{section}

\newcommand{\be}{\begin{equation}}
\newcommand{\ee}{\end{equation}}

\newcommand{\miezz}{\frac{1}{2}}

\newcommand{\eps}{\varepsilon}

\newcommand{\norm}[1]{\ensuremath{\left\Arrowvert #1 \right\Arrowvert}}
\newcommand{\norminf}[1]{\ensuremath{\left\Arrowvert #1 \right\Arrowvert_\infty}}

\newcommand{\bx}{\bar{x}}
\newcommand{\by}{\bar{y}}
\newcommand{\bz}{\bar{z}}
\newcommand{\bt}{\bar{t}}
\newcommand{\dw}{\mathbf{d}_1}
\newcommand{\seco}{\left(\delta+\frac 1\delta\left(|x-z|^4+|y-z|^4+|x+y-2z|^2\right)\right)}
\newcommand{\seca}{\left(\delta+\frac 1\delta\left(|\bx-\bz|^4+|\by-\bz|^4+|\bx+\by-2\bz|^2\right)\right)}
\newcommand{\usb}{u(\bt,\bx)+u(\bt,\by)-2u(\bt,\bz)}
\newcommand{\Gx}{\Gamma(\bx,\sigma)}
\newcommand{\Gy}{\Gamma(\by,\sigma)}
\newcommand{\Gz}{\Gamma(\bz,\sigma)}

\begin{document}

\title{A second-order Mean Field Games model with controlled diffusion}

\author{Vincenzo Ignazio}\thanks{Department of Mathematics, ETH Z\"urich. R\"amistrasse 101, 8092, Z\"urich, Switzerland. \texttt{Vincenzo.Ignazio@ksz.ch}} 


\author{Michele Ricciardi}\thanks{Dipartimento di Economia, Università LUISS Guido Carli. Viale Romania 32, Roma, Italy. \texttt{ricciardim@luiss.it}}

\date{\today}

\maketitle

\begin{abstract}
Mean Field Games (MFG) theory describes strategic interactions in differential games with a large number of small and indistinguishable players. Traditionally, the players' control impacts only the drift term in the system's dynamics, leaving the diffusion term uncontrolled. This paper explores a novel scenario where agents control both drift and diffusion. This leads to a fully non-linear MFG system with a fully non-linear Hamilton-Jacobi-Bellman equation. We use viscosity arguments to prove existence of solutions for the HJB equation, and then we adapt and extend a result from Krylov to prove a $\mathcal C^3$ regularity for $u$ in the space variable. This allows us to prove a well-posedness result for the MFG system.
\end{abstract}

\section{Introduction}
Mean Field Games theory has been introduced in 2006 by Lasry and Lions, in the pioneering articles \cite{LL1,LL2,LL-japan}, and independently in the same years by Caines, Huang and Malhamé \cite{HCM}. The theory aims to describe the asymptotic behavior of differential games involving a large number $N$ of players (also called agents), who choose their optimal strategy to minimize a certain cost functional.

Under some symmetry assumptions for the players' dynamic and the cost functionals, it has been proved that the macroscopic structure of the game, when $N\to+\infty$, can be described by a system of PDEs called \emph{Mean Field Games (MFG) system}. Here, a backward Hamilton-Jacobi-Bellman equation for the single agent's value function $u$ is coupled with a forward Fokker-Planck equation for the density evolution $m$ of the population.

To be more precise, in the asymptotic formulation the dynamic of the single agent is usually described by the following SDE:
$$
\begin{cases}
dX_s=b(X_t,\alpha_t)\,dt+\sqrt 2\sigma(X_t)\,dB_t\,,\\
X_t=x\in\R^d\,,
\end{cases}
$$
where $\alpha_\cdot$ denotes the control chosen by the single agent, chosen from a certain set of controls $\mathcal A$, $(B_t)_t$ is a $d$-dimensional Brownian motion, and $b$ and $\sigma$ are respectively the drift term and in the diffusion matrix of the process.

The cost for the single agent is usually given by the following functional:
$$
J(t,x,\alpha)=\E\left[\int_t^T\big(L(X_s,\alpha_s)\,ds+F(X_s,m(s))\big)\,ds+G(X_T,m(T))\right]\,,
$$
where $L$, $F$ and $G$ represents respectively the Lagrangian cost for the control, the running cost and the terminal cost, and where $m(\cdot)$ is given by the law of the population. In the equilibrium point we have $\mathscr L(X_t)=m(t)$ for all $t\in[0,T]$. This gives rise to the MFG system for the couple $(u,m)$:
\begin{equation}\label{standardmfg}
\begin{cases}
u_t+\mathrm{tr}(a(x)D^2u)+H(x,\nabla u)+F(x,m)=0\,,\\
m_t-\displaystyle\sum\partial^2_{ij}(a_{ij}(x)m)+\mathrm{div}(mH_p(x,\nabla u))=0\,,\\
u(T,x)=G(x,m(T))\,,\qquad m(0,x)=m_0(x)\,,
\end{cases}
\end{equation}
where $a=\sigma\sigma^*$ and where $H$ is defined by
$$
H(x,p)=\inf_{\alpha\in\mathcal A}\big\{b(x,\alpha)\cdot p+L(x,\alpha)\big\}\,.
$$
Mean Field Games have been widely studied in the literature, and most important topics such as existence of solutions, uniqueness and regularity, ergodic behavior etc... have been proved in various frameworks. We refer to \cite{cardnotes,CP-Cime,CD1,CD2} for a thorough description of the problem, both from a PDE and probabilistic approach.

So far, most of the literature mainly focuses on similar cases to the one described before, i.e. where the control of the generic agent acts only on the drift term $b$, and where the diffusion $\sigma$ is uncontrolled. But in many applied models it is highly required to consider situations where the player can also control its own diffusion. See, for instance, the works of Avellaneda et al. in \cite{avella1,avella2,avella3} for financial applications. See also \cite{hanson,Pham}.

In this paper we want to study a particular case in this direction. We consider a game where the generic agent has two controls $(\alpha,\sigma)\in\R^d\times\R$ to choose, the first one for the drift term and the second one for the diffusion matrix. This gives rise, in the cost functional, to two different Lagrangian costs $L^1$ and $L^2$, hence to two different Hamiltonians $H^1$ and $H^2$. We will be more specific about this formulation in the next section.

The Mean Field Games system we want to study takes the following form:
\begin{equation}\label{eq:mfg}
\left\{\begin{array}{ll}
\displaystyle u_t + H^1(t,x,\nabla u)+ H^2(t,x,\Delta u)+F(t,x,m)=0\,, & (t,x)\in(0,T)\times\R^d\,,\\
\displaystyle m_t -\Delta\big(mH^2_q(t,x,\Delta u)\big)+\mathrm{div}(mH_p^1(t,x,\nabla u))=0\,, & (t,x)\in(0,T)\times\R^d\,,\\
u(T,x)=G(x,m(T))\,,\qquad m(0,x)=m_0(x)\,, & x\in\R^d\,.	
\end{array}
\right.
\end{equation}
Here, $H^1(t,x,p)$ and $H^2(t,x,X)$ denote the two Hamiltonians of the system, $F$ is the running cost and $G$ the final play-off; with the notation $H_p^1$ and $H_{q_iq_j}^2$ we mean respectively the gradient and the partial derivative of $H^1$ and $H^2$ with respect to the last variable.\\

Some cases of fully non-linear Mean Field Games were studied in the literature. In \cite{dragoni} was studied an ergodic Mean Field Games with H\"ormander diffusions, and a case of Mean-Field Games with jumps was studied in \cite{dipersio}, where the control involves the drift, the diffusion and the jump size. In \cite{andrade} it is exploited a stationary fully nonlinear Mean Field Game, using a variational approach. In \cite{espen} we find a general result for parabolic fully-non linear Mean Field Games. Some Mean Field Games models belong to the cases studied both in this article and in \cite{espen}, but there are some differences: we consider $(t,x)$-depending Hamiltonians, and we restrict ourselves to the local case, whereas in \cite{espen} there is no-dependence on $(t,x)$ in the coefficients, but it is also considered the non-local case.

More extended is the literature on fractional Mean Field Games. We cite, among the others, \cite{olav,annalisa}.

As far as we know, there are no results for the well-posedness of a Mean Field Games system like \eqref{eq:mfg}. This is the aim of this paper.\\

Observe that we are considering here the case $x\in\R^d$. This is a starting point which is interesting in many applications. But in many applied models it is interesting to consider also the cases where the state dynamic of the single agent is confined into a domain $\Omega\subsetneq\R^d$. For example, the dynamic can represent the evolution of a certain capital stock, and it is natural to require $X_s\ge0$ for all $s\ge0$.

This situation can fall either in the case of Mean Field Games with Neumann boundary conditions or in the state constraint case, and we aim to consider these cases for future research. Observe that these kind of Mean Field Games have been extensively studied in the literature for uncontrolled diffusion, see e.g. \cite{cirant,Ricciardiparabolic,Porrettaweak,RicciardiNeumann,Ricciardi3} for the case of Neumann boundary conditions, and \cite{CaCa,CaCaCa,Mendico,CaMa,CapMarRic,ricciardi1, PorrettaRicciardiergodic} for the invariance condition and state constraint cases. In the fully non-linear case, the single HJB equation has been studied also for Neumann boundary conditions, see e.g. \cite{barles}.\\

Clearly, the system \eqref{eq:mfg} presents more difficulties than \eqref{standardmfg}. In the latter, the Hamilton-Jacobi-Bellman equation for the value function $u$ is semi-linear; in the former, it is fully non-linear, and the second-order coefficient of the Fokker-Planck equation strongly depends on $u$.

The well-posedness of the HJB equation will be explained in this paper, but it is not particularly difficult. These kind of equations have been widely studied in the literature, using viscosity arguments (see e.g. \cite{barles,crandall2,crandall,lionsHJB,tataru}), and while there are no specific references for a HJB equation as in \eqref{eq:mfg}, similar ideas apply to establish existence and uniqueness of solutions.

The main problem relies on the regularity of solutions. Since the coefficient for $m$ in the Laplacian term is $H^2_{q}(t,x,\Delta u)$, we must have at least $u\in W^{2,\infty}_{loc}(\R^d)$ to have existence of weak solutions for the Fokker-Planck equation (see \cite{lsu}). This cannot be achieved by viscosity arguments, since the best regularity result we can obtain is a semiconcavity estimate in the space variable.

Moreover, we have to apply a fixed-point theorem to obtain existence of solutions for \eqref{eq:mfg}. To do that, we need a $W^{1,\infty}$ norm on $H^2_q(t,x,\Delta u)$ in the space variable to prove the relative compactness of the fixed point map. This means that we need a $\mathcal C^3$ regularity for $u$ with respect to $x$. This will be done applying a regularity result of Krylov, see \cite{krylov1,krylov2}, and extending it in the case of non-smooth Hamiltonians.

We give a short summary of the results.

In Section \ref{sec2}, as already said, we give a stochastic interpretation of the system, with a process subjected to two different controls, $\alpha$ for the drift term and $\sigma$ for the diffusion term. Then we prove, in a simpler case, that the value function of this game is actually a viscosity solution of a fully nonlinear Hamilton-Jacobi equation.

In Section \ref{sec3} we start studying the Hamilton-Jacobi equation of the system \eqref{eq:mfg}, readapting the viscosity arguments exploited in \cite{crandall,crandall2} to obtain existence and uniqueness of solutions. Further we obtain, with a strengthening of the hypotheses, Lipschitz and semiconcave estimates for $u$.

In Section \ref{sec4} we restrict ourselves in a regular case, with stronger hypotheses on $H^1$ and $H^2$; applying the already mentioned result of Krylov, we obtain that $u\in\mathcal C^{1+\frac\gamma 2, 2+\gamma}$ for a certain $0<\gamma<1$, and so $u$ is actually a solution of a linear PDE. Then we can linearize the problem and obtain an estimate for $u$ in $\mathcal C^{1+\frac\alpha 2, 3+\alpha}$.

These regularity results will be essential to obtain, by approximation, a $\mathcal C^{1+\frac\alpha 2,3+\alpha}$ solution for the value function $u$ in our case of Bellman operators. This will be done in Section \ref{sec5}.
Eventually, in Section \ref{sec6} we use the regularity obtained before in order to prove existence and uniqueness of solutions for the problem \eqref{eq:mfg}, with a classical fixed point argument.

\section{From stochastic model to deterministic PDEs}\label{sec2}
Now we give a stochastic interpretation of the system \eqref{eq:mfg}.

In this framework, the generic player chooses two controls, $\alpha_\cdot$ and $\sigma_\cdot$, and plays his dynamic, described by the following process:
\begin{equation}\label{eq:dyn}
\begin{cases}
dX_s^{\alpha,\sigma}=b(X_s^{\alpha,\sigma},\alpha_s)ds + \Gamma(X_s^{\alpha,\sigma},\sigma_s) dB_s\,,\\
X_t^{\alpha,\sigma}=x\,,
\end{cases}
\end{equation}
where $x\in\R^d$, $t\in[0,T)$ and the random variable $X_s$ takes values in the whole space $\R^d$.

The time dependent control variables $\alpha_s,\sigma_s$ live in the space of bounded controls $\mathcal U\times\mathcal S$, denoting the compact sets $\mathcal{U}\subset\R^d$ and $\mathcal S\subset\R$, whereas for the continuous functions
$$
b:\R^d\times\mathcal U\to\R^d\,,\qquad\Gamma:\R^d\times\mathcal S\to\R
$$
we require that, for some constants $M>0$, $\lambda_2>\lambda_1>0$,
\begin{equation}\label{eq:bG}\begin{split}
\begin{array}{ll}
|b(x,\alpha)|\le M &\forall x\in\R^d\,,\, \alpha\in\mathcal U\,,\\
\lambda_1 <\Gamma(x,\sigma)<\lambda_2 &\forall x\in \R^d\,,\,\sigma\in\mathcal S\,.
\end{array}\end{split}\end{equation}
In particular, the bound from below for $\Gamma$ ensures the uniform ellipticity of the process. We denote the spaces of admissible controls where $\alpha$ and $\sigma$ live respectively as $\mathcal A_t^{\mathcal U}$ and $\mathcal A_t^{\mathcal S}$. From now on we omit the superscripts on the process to simplify the notations.\\

The cost function for the player is defined as 
\begin{equation*}
\mathcal J(x,t,\alpha_\cdot,\sigma_\cdot):=\E\left[\int_t^T\!\!\Big(L^1(s,X_s,\alpha_s)+L^2(s,X_s,\sigma_s)+F(s,X_s,m(s))\Big)ds + G\big(X_T,m(T)\big)\right],
\end{equation*}
where $m(s)$ is a density function, which is fixed at the moment and which will denote the density of the generic player, with initial condition $X_0=m_0.$

Our aim is to find a Hamilton-Jacobi equation, which is solved by the function
$$
u(x,t):=\inf\limits_{\substack{\alpha_\cdot\in\mathcal A_t^{\mathcal U}\\ \sigma_\cdot\in\mathcal A_t^{\mathcal S}}}\mathcal J(x,t,\alpha_\cdot,\sigma_\cdot)\,.
$$
For this we use the dynamic programming principle written in the following form:
\begin{equation}\label{eq:dpp}
\begin{split}
u(x,t)=\inf\limits_{\substack{\alpha_\cdot\in\mathcal A_t^{\mathcal U}\\ \sigma_\cdot\in\mathcal A_t^{\mathcal S}}} \E\left[\int_t^{t+h}\!\!\Big(L^1(s,X_s,\alpha_s)\right. &+L^2(s,X_s,\sigma_s)\\
&+F(s,X_s,m(s))\left.\Big)ds + v(X_{t+h},t+h)\right]\,.
\end{split}
\end{equation}
The proof of the verification theorem in this general case is quite technical and can be found in \cite{soner}. We will give here a direct proof in a simpler case, where the process of the generic player follows the equation
$$
\begin{cases}
dX_s^{\alpha,\sigma}=\alpha_s\,ds+\sigma_s\,dB_s\\
X_t^{\alpha,\sigma}=x\,,
\end{cases}
$$
where the time dependent control variables $\alpha_s$, $\sigma_s$ live in the space of bounded controls $\mathcal U\times\mathcal S$, and the matrices in $\mathcal S$ are bounded from below (and from above) by a positive constant, as in \eqref{eq:bG}.

The following lemma will come in handy in the proof of the main result:
\begin{lem}\label{lem:handy}
Let $t,r\in[0,T]$, $x,y\in\R^d$. We consider the two processes
\begin{equation}\label{eq:dinXY}
\begin{array}{ll}
\begin{cases}
dX_s=\alpha_s\,ds+\sigma_s\,dB_s\,,\\
X_t=x\,,
\end{cases}&\begin{cases}
dY_s=\beta_s\,ds+\eta_s\,dB_s\,,\\
Y_r=y\,,
\end{cases}\end{array}
\end{equation}
with $X_s=x$ for $s<t$ and $Y_s=y$ for $s<r$, and where $\alpha_\cdot$,$\beta_\cdot\in\mathcal A_t^{\mathcal U}$, $\sigma_\cdot,\eta_\cdot\in\mathcal A_t^{\mathcal S}$ are bounded processes in $[0,T]$.

Let $h>0$. Then there exists a constant $C$ (not depending on $h$) such that
\begin{equation}\label{eq:traiettorievicine}
\E\left[\sup\limits_{t\le s\le t+h}|X_s-Y_s|\right]\le C\left(|x-y|+\sqrt{|t-r|}+M\sqrt h\right)\,,
\end{equation}
where $C$ depends on $\norminf{\alpha}$, $\norminf{\beta}$, $\norminf{\sigma}$, $\norminf{\eta}$ and where $M:=\norminf{\alpha-\beta}+\norminf{\eta-\sigma}\,.$

In particular, we have
\begin{equation}\label{eq:stabilita}
\E\left[\sup\limits_{t\le s\le t+h} |X_s-x|\right]\le C\sqrt h\,.
\end{equation}
\end{lem}
\begin{proof}
Without loss of generality, we suppose $t>r$. We start noticing that \eqref{eq:stabilita} is a directly consequence of \eqref{eq:traiettorievicine}. Actually, taking a process $(X_s)_s$ satisfying \eqref{eq:dinXY}, the constant process $Y_s=x$ satisfies \eqref{eq:dinXY} with $r=t$, $y=x$, $\beta_s=0$, $\eta_s=0$. Then \eqref{eq:traiettorievicine} tells us that
$$
\E\left[\sup\limits_{t\le s\le t+h} |X_s-x|\right]=	\E\left[\sup\limits_{t\le s\le t+h}|X_s-Y_s|\right]\le C\sqrt h\,.
$$
Hence, we only have to prove \eqref{eq:traiettorievicine}. Since $t>r$, we can write
$$
X_s-Y_s=x+\int_t^s\big[(\alpha_u-\beta_u)\,du+(\sigma_u-\eta_u)\,dB_u\big]-y-\int_r^t\big[\beta_u\,du+\eta_u\,dB_u\big]\,.
$$
Using the boundedness of $\alpha$ and $\beta$, we get for $s\in[t,t+h]$
$$
|X_s-Y_s|\le|x-y|+Mh+C(t-r)+\left|\int_t^s(\sigma_u-\eta_u)\,dB_u\right|+\left|\int_r^t\eta_u\,dB_u\right|\,,
$$
and so
\begin{equation}\label{eq:prebaldi}
\begin{split}
\E\left[\sup\limits_{t\le s\le t+h}|X_s-x|\right]&\le |x-y|+C(t-r)+Ch\\
&+\E\left[\sup\limits_{t\le s\le t+h}\left|\int_t^s(\sigma_u-\eta_u)\,dB_u\right|\right]+\E\left[\left|\int_r^t\eta_u\,dB_u\right|\right]\,.
\end{split}	
\end{equation}
The last two terms in the right-hand side are estimated using \emph{Proposition 8.5} and \emph{Proposition 8.6} of \cite{baldi}:
\begin{align*}
\E\left[\sup\limits_{t\le s\le t+h}\left|\int_t^s(\sigma_u-\eta_u)\,dB_u\right|\right]&\le \sqrt{\E\left[\sup\limits_{t\le s\le t+h}\left|\int_t^s(\sigma_u-\eta_u)\,dB_u\right|^2\right]}\\
&\le C\sqrt{\E\left[\int_t^{t+h}(\sigma_u-\eta_u)^2\,du\right]}\le CM\sqrt h\,.
\end{align*}
In an easier way we estimate the last term:
$$
\E\left[\left|\int_r^t\eta_u\,dB_u\right|\right]\le\sqrt{\E\left[\left|\int_r^t\eta_u\,dB_u\right|^2\right]}=\sqrt{\E\left[\int_r^t\eta_u^2\,du\right]}\le C\sqrt{t-r}\,,
$$
Plugging these estimates in \eqref{eq:prebaldi}, we obtain \eqref{eq:traiettorievicine}\,.
\end{proof}

Now we are able to prove the following theorem.
\begin{thm}\label{thm:verification}
Suppose that $L^1$, $L^2$, $F$ and $G$ are Lipschitz continuous in the space variable and globally bounded. Furthermore, suppose that $L^1$ and $L^2$ are continuous in time and space, and that the function $(t,x)\to F(t,x,m(t))$ is continuous in time and space.

Then $u$ is a viscosity solution of the Hamilton-Jacobi equation
\begin{equation}\label{hjb}
\begin{cases}
u_t(t,x)+H^1(t,x,\nabla u(t,x))+H^2(t,x,\Delta u(t,x))+F(t,x,m(t))=0\,,\\
u(T,x)=G(x,m(T))\,,
\end{cases}
\end{equation}
where
$$
H^1(t,x,p):=\inf\limits_{\alpha\in\mathcal U}\{\langle p,\alpha\rangle + L^1(t,x,\alpha)\}
$$
and
\begin{equation}\label{eq:L2}
H^2(t,x,q):=\inf\limits_{\sigma\in\mathcal S}\left\{\miezz\sigma^2q+L^2(t,x,\sigma)\right\}
\end{equation}
are the so called Hamiltonians.
\end{thm}
We give an example of a function $F$ satisfying the previous assumptions. Let $\mathcal P(\R^d)$ be the space of Borel probability measures with finite first order moment, equipped with the Wasserstein distance $\mathbf{d}_1$ defined in the following way: for $m_1\,,m_2\in\mathcal{P}(\R^d)$
\begin{equation}\label{def:wass}
\dw(m_1,m_2):=\sup\limits_{Lip(\phi)= 1}\int_{\R^d} \phi(x)d(m_1-m_2)(x)\,.
\end{equation}
If $m:[0,T]\to\mathcal{P}(\R^d)$ is continuous, then a function $F:[0,T]\times\R^d\times\mathcal{P}(\R^d)\to\R\,$, continuous in all variables and Lipschitz in space variable, satisfies the required assumptions.
\begin{proof}
\emph{Step 1: Continuity of $u$.} The first step consists to check the continuity of the value function $u$ in both variables.

Let $(t_n,x_n)\to(t,x)$. Using the definition of $u$, for each $\eps>0$ we considers controls $\alpha_s^{\eps,n}$ and $\sigma_s^{\eps,n}$ such that
\begin{align*}
\E\left[\int_{t_n}^T\Big(L^1(s,X_s^{\eps,n},\alpha_s^{\eps,n})+L^2(s,X_s^{\eps,n},\sigma_s^{\eps,n})+F(s,X_s^{\eps,n},m(s))\Big)\,ds+G(X_T^{\eps,n},m(T))\right]\\\le u(t_n,x_n)+\eps\,,
\end{align*}
where $X_s^{\eps,n}$ is the process related to the controls $\alpha_s^{\eps,n}$ and $\sigma_s^{\eps,n}$, with $X_s^{\eps,n}=x_n$ for $s<t_n$.

We take the process $(X_s)_s$ defined $X_s=x$ for $s<t$ and satisfying for $s\ge t$
$$
\begin{cases}
dX_s=\alpha_s^{\eps,n}\,ds+\sigma_s^{\eps,n}\,dB_s\,,\\
X_t=x\,.
\end{cases}
$$
Then we have, using the hypotheses on the cost functions and the Lagrangians,
\begin{align*}
u(t,x)&\le \E\left[\int_t^T\Big(L^1(s,X_s,\alpha_s^{\eps,n})+L^2(s,X_s,\sigma_s^{\eps,n})+F(s,X_s,m(s))\Big)\,ds+G(X_T,m(T))\right]\\
&\le\E\left[\int_{t_n}^T\Big(L^1(s,X_s,\alpha_s^{\eps,n})+L^2(s,X_s,\sigma_s^{\eps,n})+F(s,X_s,m(s))\Big)\,ds+G(X_T,m(T))\right]\\
&+C|t_n-t|\le C\,\E\left[\sup\limits_{t\in[0,T]}|X_s-X_s^{\eps,n}|\right]+C|t_n-t|+u(t_n,x_n)+\eps\\
&\le C\left(|x_n-x|+\sqrt{|t_n-t|}\right)+u(t_n,x_n)+\eps\,,
\end{align*}
where we used in the last step \eqref{eq:traiettorievicine} with $M=0$.

So, passing to the $\liminf$ when $(t_n,x_n)\to(t,x)$, we obtain
$$
u(t,x)\le \liminf\limits_{(t_n,x_n)\to(t,x)} u(t_n,x_n)+\eps\,,
$$
which gives, for the arbitrariness of $\eps$,
$$
u(t,x)\le\liminf\limits_{(t_n,x_n)\to(t,x)}u(t_n,x_n)\,.
$$
In the same way, we prove $u(t,x)\ge\limsup\limits_{(t_n,x_n)\to(t,x)} u(t_n,x_n)$, which finally gives the continuity of $u$.

\emph{Step 2: Subsolution argument}. Let $\varphi\in\mathcal C^\infty([0,T]\times\R^d)$, with all derivatives bounded, a test function with
$$
\varphi(\bt,\bx)\ge u(\bt,\bx)\,,\qquad\forall(\bt,\bx)\in[0,T]\times\R^d
$$
and a minimum value at the point $(t,x)$, where $\varphi(t,x)=u(t,x)$. The dynamic programming principle \eqref{eq:dpp} gives
\begin{small}
\begin{equation}\label{eq:subdpp}
\begin{split}
0&=\inf\limits_{\substack{\alpha_\cdot\in\mathcal A_t^{\mathcal U}\\\sigma_\cdot\in\mathcal A_t^{\mathcal S}}}\E\left[\int_t^{t+h}\Big(L^1(s,X_s,\alpha_s)+L^2(s,X_s,\sigma_s)+F(s,X_s,m(s))\Big)\,ds+u(t+h,X_{t+h})-u(t,x)\right]\\
&\le\inf\limits_{\substack{\alpha_\cdot\in\mathcal A_t^{\mathcal U}\\\sigma_\cdot\in\mathcal A_t^{\mathcal S}}}\E\left[\int_t^{t+h}\Big(L^1(s,X_s,\alpha_s)+L^2(s,X_s,\sigma_s)+F(s,X_s,m(s))\Big)\,ds+\varphi(t+h,X_{t+h})-\varphi(t,x)\right].
\end{split}
\end{equation}
\end{small}
We now use Ito's formula on $\varphi$ and obtain, for any admissible control $(\alpha_s,\sigma_s)$,
$$
\E\big[\varphi(t+h,X_{t+h})-\varphi(t,x)\big]=\E\left[\int_t^{t+h}\left(\varphi_t+\langle\nabla\varphi,\alpha_s\rangle+\miezz\sigma_s^2\Delta\varphi\right)(s,X_s)\,ds\right]\,.
$$
This expression, combined with \eqref{eq:subdpp}, gives us, for each control $\alpha_\cdot$ and $\sigma_\cdot$,
\begin{align*}
0&\le \E\left[\int_t^{t+h}\left\{\varphi_t(s,X_s)+\langle\nabla\varphi(s,X_s),\alpha_s\rangle+\frac{\sigma_s^2}2\Delta\varphi(s,X_s)\right\}ds\right]\\
&+\E\left[\int_t^{t+h}\Big\{L^1(s,X_s,\alpha_s)+L^2(s,X_s,\sigma_s)+F(s,X_s,m(s))\Big\}\,ds\right]\,.
\end{align*}
Now it is necessary to write down our PDE as a purely deterministic expression. For this, we consider for now only a subset of the control spaces $\mathcal A_t^{\mathcal U}$ and $\mathcal A_t^{\mathcal S}$, namely the set of constant controls. This controls are denoted by $\alpha$ and $\sigma$ instead of $\alpha_s$ and $\sigma_s$, taking values in $\mathcal U\subset\R^d$ and $\mathcal S\subset\R$. Our inequality now reads
\begin{align*}
0&\le \E\left[\int_t^{t+h}\left\{\varphi_t(s,X_s)+\langle\nabla\varphi(s,X_s),\alpha\rangle+\frac{\sigma^2}2\Delta\varphi(s,X_s)\right\}ds\right]\\
&+\E\left[\int_t^{t+h}\Big\{L^1(s,X_s,\alpha)+L^2(s,X_s,\sigma)+F(s,X_s,m(s))\Big\}\,ds\right]\,.
\end{align*}
Now we divide by $h$ and we let $h\to 0$. Using the $a.s.$ time continuity of the trajectories $X_\cdot(\omega)$ and the continuity of the cost functions  and the Lagrangians, we obtain
$$
0\le\varphi_t(t,x)+\langle \nabla\varphi(t,x),\alpha\rangle+\frac{\sigma^2}2\Delta\varphi(t,x)+L^1(t,x,\alpha)+L^2(t,x,\sigma)+F(t,x,m(t))\,.
$$
Passing to the $\inf$ when $\alpha\in\mathcal U$ and $\sigma\in\mathcal S$, we get
$$
0\ge-\varphi_t(t,x)-H^1(t,x,\nabla\varphi(t,x))-H^2(t,x,\Delta\varphi(t,x))-F(t,x,m(t,x))
$$
for all test functions $\varphi$, which proves that $u(t,x)$ is indeed a viscosity subsolution for our PDE.

\emph{Step 3: Supersolution argument}. Now we take $\phi\in\mathcal C_b^\infty([0,T]\times\R^d)$, a test function with
$$
\phi(\bt,\bx)\le u(\bt,\bx)\,,\qquad\forall (\bt,\bx)\in[0,T]\times\R^d
$$
and a maximum value at the point $(t,x)$, where $\phi(t,x)=u(t,x)$. As before, the dynamic programming principle \eqref{eq:dpp} gives
$$
0\ge\inf\limits_{\substack{\alpha_\cdot\in\mathcal A_t^{\mathcal U}\\\sigma_\cdot\in\mathcal A_t^{\mathcal S}}}\E\left[\int_t^{t+h}\Big(L^1(s,X_s,\alpha_s)+L^2(s,X_s,\sigma_s)+F(s,X_s,m(s))\Big)\,ds+\phi(t+h,X_{t+h})-\phi(t,x)\right].
$$
Using the definition of $\phi$, we take controls $\alpha_s^h$, $\sigma_s^h$ and the related process $X_s^h$ such that
$$
E\left[\int_t^{t+h}\Big(L^1(s,X_s^h,\alpha_s^h)+L^2(s,X_s^h,\sigma_s^h)+F(s,X_s^h,m(s))\Big)\,ds+\phi(t+h,X_{t+h}^h)-\phi(t,x)\right]\le h^2\,.
$$
Applying Ito's formula we obtain
\begin{align*}
&\E\left[\int_t^{t+h}\left\{\phi_t(s,X_s^h)+\langle\nabla\phi(s,X_s^h),\alpha_s^h\rangle+\frac{{(\sigma_s^h)}^2}2\Delta\phi(s,X_s^h)\right\}ds\right]\\
+&\E\left[\int_t^{t+h}\Big\{L^1(s,X_s^h,\alpha_s^h)+L^2(s,X_s^h,\sigma_s^h)+F(s,X_s^h,m(s))\Big\}\,ds\right]\le h^2\,.
\end{align*}
We estimate all the terms in the same way, using the Lipschitz continuity of the cost functions with respect to $x$ and Lemma \ref{lem:handy}. For instance, for the $L^1$ term we get
\begin{align*}
&\E\left[\int_t^{t+h}L^1(s,X_s^h,\alpha_s^h)\,ds\right]\ge\E\left[\int_t^{t+h}\Big(L^1(s,x,\alpha_s^h)-|X_s^h-x|\Big)\,ds\right]\\
\ge&\E\left[\int_t^{t+h} L^1(s,x,\alpha_s^h)\,ds\right]-h\,\E\left[\sup\limits_{t\le s\le t+h}|X_s^h-x|\right]\ge\E\left[\int_t^{t+h}L^1(s,x,\alpha_s^h)\right]-Ch\sqrt h\,.
\end{align*}
Hence, with all these estimates we get
\begin{align*}
&\E\left[\int_t^{t+h}\left\{\phi_t(s,x)+\langle\nabla\phi(s,x),\alpha_s^h\rangle+\frac{{(\sigma_s^h)}^2}2\Delta\phi(s,x)\right\}ds\right]\\
+&\E\left[\int_t^{t+h}\Big\{L^1(s,x,\alpha_s^h)+L^2(s,x,\sigma_s^h)+F(s,x,m(s))\Big\}\,ds\right]\le h^2+Ch\sqrt h\le Ch\sqrt h\,.
\end{align*}
Using the definition of $H^1$ and $H^2$, we obtain
$$
\E\left[\int_t^{t+h}\!\!\Big\{\phi_t(s,x)+H^1(s,x,\nabla\phi(s,x))+H^2(s,x,\Delta\phi(s,x))+F(s,x,m(s))\Big\}\,ds\right]\le Ch\sqrt h\,.
$$
Now we divide by $h$ and we use the continuity of $H^1$ and $H^2$ in time variable, which is an immediate consequence of the time continuity of $L^1$ and $L^2$. Eventually, we obtain
$$
-\phi_t(t,x)-H^1(t,x,\nabla\phi)-H^2(t,x,\Delta\phi(t,x))-F(t,x,m(t,x))\ge0\,,
$$
which proves that $u$ is a supersolution of the HJB equation and concludes the theorem.
\end{proof}
\begin{rem}\label{rem}
We observe that, up to defining the set $\mathcal S':=\left\{\frac{\sigma^2}2\,\Big|\,\sigma\in\mathcal S\right\}$, the function $H^2$ can be rewritten as
$$
H^2(t,x,q):=\inf\limits_{\eta\in\mathcal{S'}}\big\{\eta q+L_3(t,x,\eta)\big\}\,,
$$
where $L_3(t,x,\eta)=L^2(t,x,\sqrt{2\eta})$.

Henceforth, when we will talk about this simplified model case, we will refer to
\begin{equation}\label{eq:L3}
\begin{split}
H^1(t,x,p)&=\inf\limits_{\alpha\in\mathcal U}\big\{\langle p,\alpha\rangle + L^1(t,x,\alpha)\big\}\,,\\
H^2(t,x,q)&=\inf\limits_{\eta\in\mathcal{S'}}\big\{\eta q+L_3(t,x,\eta)\big\}\,.
\end{split}
\end{equation}
\end{rem}

\section{The Hamilton-Jacobi Equation}\label{sec3}

This Section is completely devoted to the study of the following equation:
$$
\left\{
\begin{array}{lr}
u_t+H^2(t,x,\Delta u)+H^1(t,x,\nabla u)=-F(t,x)\,, &(t,x)\in[0,T]\times\R^d\,,\\
u(T,x)=G(x)\,,& x\in\R^d\,,
\end{array}
\right.
$$
where
$$
H^1(t,x,p):=\inf\limits_{\alpha\in\mathcal U}\big\{ \langle p,b(x,\alpha)\rangle + L^1(t,x,\alpha) \big\}
$$
and
$$
H^2(t,x,q):=\inf\limits_{\sigma\in\mathcal S}\left\{ \frac{\Gamma(x,\sigma)^2}2q + L^2(t,x,\sigma) \right\}\,.
$$
Since there is no dependence on $m$ in this step, we can omit $F$ in the previous equation, including it into $H^1$ (up to changing the Lagrangian $L^1$),. Hence, the equation we want to study is the following:
\begin{equation}\label{eq:hjb}
\begin{cases}
u_t+H^2(t,x,\Delta u)+H^1(t,x,\nabla u)=0\,, \qquad(t,x)\in[0,T]\times\R^d\,,\\
u(T,x)=G(x)\,,\hspace{6.75cm} x\in\R^d\,.
\end{cases}
\end{equation}

As already said, the results obtained in this section are readaptation of the results obtained in the elliptic case in \cite{crandall,crandall2} and in the literature of viscosity results. We decide to include them for the sake of completeness, since, as far as we know, there are no results directly applicable to \eqref{eq:hjb}.\\

We note that, with the change of variable $v(t,x)=e^{-\lambda(T-t)}u(t,x)$, the system \eqref{eq:hjb} is equivalent to the following one:
\begin{equation}\label{eq:hjb2}
\begin{cases}
-u_t-H^2_\lambda(t,x,\Delta u)-H^1_\lambda(t,x,\nabla u)+\lambda u=0\,, \qquad(t,x)\in[0,T]\times\R^d\,,\\
u(T,x)=G(x)\,,\hspace{8.1cm} x\in\R^d\,,
\end{cases}
\end{equation}
where
$$
\begin{array}{l}
H^2_\lambda(t,x,q)=e^{-\lambda (T-t)}H^2(t,x,e^{\lambda(T-t)}q)=\inf\limits_{\sigma\in\mathcal S}\left\{ \frac{\Gamma(x,\sigma)^2}2q + e^{-\lambda(T-t)}L^2(t,x,\sigma) \right\}\,,\\
H^1_\lambda(t,x,p)=e^{-\lambda (T-t)}H^1(t,x,e^{\lambda(T-t)}p)=\inf\limits_{\alpha\in\mathcal U}\Big\{ \langle p,b(x,\alpha)\rangle + e^{-\lambda(T-t)}L^1(t,x,\alpha) \Big\}\,,
\end{array}
$$
for each $\lambda>0$.\\

So, in order to prove existence, uniqueness and regularity of solutions, we will work with \eqref{eq:hjb2}. For a better readability, we will write $H^1$ and $H^2$ instead of $H^1_\lambda$ and $H^2_\lambda$, up to changing again the Lagrangian functions $L^1$ and $L^2$.

Due to the non-linearity of the second-order term, we need to work in viscosity sense. First, for each $u:(0,T)\times\R^d\to\R$ we define the following sets:
\begin{defn}
We denote with $\mathcal P^{2,+} u(t_0,x_0)$ the set of all points $(a,p,X)\in\R\times\R^d\times Sym(d,d)$ such that for $(t,x)\to(t_0,x_0)$ we have
$$
u(t,x)\le u(t_0,x_0)+a(t-t_0)+\langle p,x-x_0\rangle+\miezz \big\langle X(x-x_0),x-x_0\big\rangle + o\big(|t-t_0|+|x-x_0|^2\big)\,.
$$
Similary, we define $\mathcal P^{2,-} u(t_0,x_0)=-\mathcal P^{2,+}(-u)(t_0,x_0)$, i.e. the set of all points $(a,p,X)\in\R\times\R^d\times Sym(d,d)$ such that for $(t,x)\to(t_0,x_0)$ we have
$$
u(t,x)\ge u(t_0,x_0)+a(t-t_0)+\langle p,x-x_0\rangle+\miezz \big\langle X(x-x_0),x-x_0\big\rangle + o\big(|t-t_0|+|x-x_0|^2\big)\,.
$$
\end{defn}
Now we give a suitable definition of solution.
\begin{defn}
We say that a function $u\in USC((0,T]\times\R^d)$ and bounded from above is a subsolution of \eqref{eq:hjb2} if $\forall(t,x)\in(0,T)\times\R^d$ and $(a,p,X)\in\mathcal P^{2,+}u(t,x)$ we have
$$
-a-H^2(t,x,\mathrm{tr}(X))-H^1(t,x,p)+\lambda u(t,x)\le0\,;
$$
moreover, $\forall x\in\R^d$ we require $u(T,x)\le G(x)$.

Similarly, we say that a function $u\in LSC((0,T]\times\R^d)$ and bounded from below is a supersolution of \eqref{eq:hjb2} if $\forall(t,x)\in(0,T)\times\R^d$ and $(a,p,X)\in\mathcal P^{2,-}u(t,x)$ we have
$$
-a-H^2(t,x,\mathrm{tr}(X))-H^1(t,x,p)+\lambda u(t,x)\ge0\,;
$$
moreover, $\forall x\in\R^d$ we require $u(T,x)\le G(x)$.

Finally, we say that a bounded continuous function $u$ is a solution of \eqref{eq:hjb2} if it is both a subsolution and supersolution.
\end{defn}

The first step is to prove a comparison principle for \eqref{eq:hjb2}, and for that the new term $\lambda u$ and the boundedness hypotheses play an essential role.

In order to prove it, we also need the following proposition (Theorem 8.3 of \cite{crandall2}), which will be also useful in the rest of the section.
\begin{prop}\label{prop:importante}
Let $u_1,\dots,u_n$ subsolutions of \eqref{eq:hjb2}.

Consider a function $\phi:(0,T)\times\R^{nd}\to\R$, once continuously differentiable in $t$ and twice continuously differentiable in $(x_1,\dots,x_n)$.

Suppose that
$$
u_1(t,x_1)+\ldots+u_n(t,x_n)-\phi(t,x_1,\dots,x_n)
$$
achieves his maximum in $(\tilde{t},\tilde{x_1},\dots,\tilde{x_n})$. Then there exist $a_i\in\R$, $X_i\in Sym(d,d)$, $i=1,\ldots,n$ such that, defining $A=D^2_x\phi(\tilde{t},\tilde{x_1},\dots,\tilde{x_n})$, one has
$$
\begin{array}{cc}
(a_i,\nabla_{x_i}\phi(\tilde{t},\tilde{x_1},\ldots,\tilde{x_n}),X_i)\in\overline{\mathcal P^{2,+}u_i(\tilde{t},\tilde{x_i})}\,, & \sum\limits_i a_i=\phi_t(\tilde{t},\tilde{x_1},\ldots,\tilde{x_n})\,,\\
\begin{pmatrix}
X_1 & \dots & 0\\
\vdots & \ddots & \vdots\\
0 & \dots & X_N
\end{pmatrix}\le A\,.
\end{array}
$$
\end{prop}
Now we are ready to prove the comparison principle, by adapting the proof given by Crandall, Ishii and Lions in \cite{crandall}.
\begin{thm}\label{thm:comparison}
Suppose \eqref{eq:bG} is satisfied. Moreover, suppose that
\begin{itemize}
\item $L^1$ continuous in all variables and H\"older in $x$, uniformly in $t\in[0,T]$ and locally uniformly in $p$, i.e. $\forall |p|\le L$ $\exists K_L$ such that
$$
|L^1(t,x,p)-L^1(t,y,p)|\le K_L|x-y|^\beta\,,\qquad \mbox{for a certain }0<\beta\le 1\,,\quad\forall\,t\in[0,T]\,.
$$
The same hypotheses hold for $L^2$;
\item $\Gamma$, $b$ and $G$ are Lipschitz in the space variable.
\end{itemize}
Then the comparison principle holds for equation \eqref{eq:hjb2}, i.e., if $u$ and $v$ are respectively a subsolution and a supersolution of \eqref{eq:hjb2}, then $u(t,x)\le v(t,x)$ for all $(t,x)\in(0,T]\times\R^d$\,.
\end{thm}
\begin{proof}
Suppose, by contradiction, that $\exists\,(s,z)\in(0,T]\in\R^d$ such that $u(s,z)-v(s,z)=\delta>0$.

We consider, for $\alpha,\nu>0$, the following quantity
\begin{equation}\label{eq:doubling}
u(t,x)-v(t,y)-\frac\alpha2 |x-y|^2-\frac 1\alpha |x|^2-\frac\nu t\,.
\end{equation}
Due to the coercive term $\frac 1\alpha2|x|^2$ and the boundedness of $u$ and $v$, we know that \eqref{eq:doubling} achieves a maximum.\\

We denote the maximum by $M$ and one of its maximum points with $(\bt,\bx,\by)\in(0,T]\in\R^{2d}$.

We must have, for $\nu$ sufficiently small and $\alpha$ sufficiently large,
\begin{equation}\label{eq:unastima}
u(\bt,\bx)-v(\bt,\by)-\frac\alpha 2|\bx-\by|^2-\frac 1\alpha|\bx|^2-\frac\nu\bt\ge u(s,z)-v(s,z)-\frac 1\alpha|z|^2-\frac\nu s\ge\delta-\frac\delta 2=\frac\delta 2\,.
\end{equation}
This implies, thanks to the boundedness of $u$ and $v$,
$$
\frac 1\alpha|\bx|^2+\frac\alpha 2|\bx-\by|^2\le u(\bt,\bx)-v(\bt,\by)\le C\implies\lim\limits_{\alpha\to+\infty}|\bx-\by|=0\,.
$$
This cannot happen in $\bt=T$. In this case we have
$$
\frac\delta 2\le u(s,z)-v(s,z)-\frac 1\alpha|z|^2-\frac\nu s\le u_T(\bx)-v_T(\by)-\frac\alpha 2|\bx-\by|^2-\frac 1\alpha|\bx|^2\,.
$$
Since $G$ is Lipschitz, $u_T\le v_T$ and $|\bx-\by|\to0$, we have for $\alpha$ sufficiently large
$$
u_T(\bx)-v_T(\by)\le\frac\delta 3\implies u_T(\bx)-v_T(\by)-\frac\alpha 2|\bx-\by|^2-\frac 1\alpha|\bx|^2\le\frac\delta 3\,,
$$
which gives a contradiction.

Then we must have $\bt\in(0,T)$. Applying Proposition \ref{prop:importante} with
\begin{align*}
&u_1(t,x)=u(t,x)\,,\qquad u_2(t,y)=-v(t,y)\,,\\
&\phi(t,x,y)=\frac\alpha 2|x-y|^2+\frac 1\alpha|x|^2+\frac\nu t\,,
\end{align*}
we obtain that there eists $a,b,X,Y$ such that
$$
(a,\nabla_x\phi(\bt,\bx,\by),X)\in\overline{\mathcal P^{2,+}u(\bt,\bx)}\,,\qquad (-b,-\nabla_y\phi(\bt,\bx,\by),-Y)\in\overline{\mathcal P^{2,-}v(\bt,\by)}\,,
$$
and
\begin{equation}\label{eq:checoglia}
a+b=-\frac\nu{\bt^2}\,,\qquad\begin{pmatrix}
X & 0\\
0 & Y
\end{pmatrix}\le D^2\phi(\bt,\bx,\by)\,.
\end{equation}
From now on, we will omit for the function $\phi$ his dependence on $(\bt,\bx,\by)$.

Since $u$ is a subsolution and $v$ a supersolution, one has
\begin{align*}
-&a-H^2(\bt,\bx,\mathrm{tr}(X))-H^1(\bt,\bx,\nabla_x\phi)+\lambda u(\bt,\bx)\le 0\,,\\
&b-H^2(\bt,\by,-\mathrm{tr}(Y))-H^1(\bt,\by,-\nabla_y\phi)+\lambda v(\bt,\by)\ge0\,.
\end{align*}
Subtracting the two inequalities we obtain
\begin{equation}\label{eq:nonhoidee}
\begin{split}
\frac\nu{\bt^2}+\lambda(u(\bt,\bx)-v(\bt,\by))&\le H^2(\bt,\bx,\mathrm{tr}(X))-H^2(\bt,\by,-\mathrm{tr}(Y))\\
&H^1(\bt,\bx,\nabla_x\phi)-H^1(\bt,\by,-\nabla_y\phi)\,.
\end{split}
\end{equation}
The first term in the left-hand side is non-negative, so we can ignore it. For the second term, we use \eqref{eq:unastima} to get
\begin{align*}
\frac\delta 2+\lambda\frac\alpha 2|\bx-\by|^2&\le H^2(\bt,\bx,\mathrm{tr}(X))-H^2(\bt,\by,-\mathrm{tr}(Y))\\
&+H^1(\bt,\bx,\nabla_x\phi)-H^1(\bt,\by,-\nabla_y\phi)\,.
\end{align*}
In order to estimate the last two terms of the right-hand side term, we first compute the derivatives of $\phi$. We have
\begin{align*}
&\nabla_x\phi=\alpha(\bx-\by)+\frac2\alpha\bx\,,\qquad\nabla_y\phi=-\alpha(\bx-\by)\,,\\
&D^2\phi=\alpha\begin{pmatrix}
I & -I\\
-I & I
\end{pmatrix}+\frac 2\alpha\begin{pmatrix}
I & 0\\
0 & 0
\end{pmatrix}\,.
\end{align*}
Using the very definition of $H^2$, we obtain, calling $\sigma$ the optimal control for $H^2(\bt,\by,-\mathrm{tr}(Y))$,
\begin{equation}\label{eq:H2}\begin{split}
& H^2(\bt,\bx,\mathrm{tr}(X))-H^2(\bt,\by,-\mathrm{tr}(Y))\\
\le&\frac{\Gamma(\bx,\sigma)^2}{2}\mathrm{tr}(X)+\frac{\Gamma(\by,\sigma)^2}{2}\mathrm{tr}(Y)+e^{-\lambda(T-t)}\big(L^2(\bt,\bx,\sigma)-L^2(\bt,\by,\sigma)\big)\\
\le&C|\bx-\by|^\beta+\mathrm{tr}\left(\frac{\Gamma(\bx,\sigma)^2}{2}X+\frac{\Gamma(\by,\sigma)^2}{2}Y\right)\le\omega(\alpha)+\mathrm{tr}\left(\frac{\Gamma(\bx,\sigma)^2}{2}X+\frac{\Gamma(\by,\sigma)^2}{2}Y\right)\,.
\end{split}\end{equation}
The last term is estimated as follows.

We have
$$
\mathrm{tr}\left(\frac{\Gamma(\bx,\sigma)^2}{2}X+\frac{\Gamma(\by,\sigma)^2}{2}Y\right)=\miezz\mathrm{tr}\left(B\begin{pmatrix}
X & 0\\
0 & Y
\end{pmatrix}\right)\,,
$$
where $B$ is equal to
\begin{equation}\label{eq:miservedopo}
\begin{pmatrix}
\Gamma(\bx,\sigma)^2 & \Gamma(\bx,\sigma)\Gamma(\by,\sigma)\\
\Gamma(\bx,\sigma)\Gamma(\by,\sigma) & \Gamma(\by,\sigma)^2
\end{pmatrix}\,.
\end{equation}
Since $B$ is non-negative definite, we can use \eqref{eq:checoglia} to obtain
\begin{align*}
&\mathrm{tr}\left(\frac{\Gamma(\bx,\sigma)^2}{2}X+\frac{\Gamma(\by,\sigma)^2}{2}Y\right)\le\mathrm{tr}(BD^2\phi)\\
=\,&\alpha d|\Gamma(\bx,\sigma)-\Gamma(\by,\sigma)|^2+\frac{2d}\alpha\Gamma(\bx,\sigma)^2\le C\alpha|\bx-\by|^2+\omega(\alpha)\,,
\end{align*}
where $\omega(\alpha)$ is a quantity depending on $\alpha$ such that $\omega(\alpha)\to 0$ when $\alpha\to+\infty$.

We argue in a similar way in order to bound the $H^1$ term. We have, calling $\alpha_0$ the optimal control for $H^1(\bt,\by,-\nabla_y\phi)$,
\begin{equation}\label{eq:H1}\begin{split}
&H^1(\bt,\bx,\nabla_x\phi)-H^1(\bt,\by,\nabla_y\phi)\\
\le\,&\langle b(\bx,\alpha_0),\nabla_x\phi\rangle+\langle b(\by,\alpha_0),\nabla_y\phi\rangle +e^{-\lambda(T-t)}\big(L^1(\bt,\bx,\alpha_0)-L^1(\bt,\by,\alpha_0)\big)\\
\le\,&|\bx-\by|^\beta+\langle b(\bx,\alpha_0)-b(\by,\alpha_0),\alpha(\bx-\by)\rangle+\frac 2\alpha\langle b(\bx,\alpha_0),\bx\rangle\le C\alpha|\bx-\by|^2+\omega(\alpha)\,,
\end{split}\end{equation}
where in the last passage we used the Lipschitz bound and the boundedness of $b$.

Putting together all the estimates in \eqref{eq:nonhoidee}, we obtain
$$
\frac\delta 2+\lambda\frac\alpha 2|\bx-\by|^2\le C\alpha|\bx-\by|^2+\omega(\alpha)\,,
$$
which gives a contradiction for $\alpha$ sufficiently small and $\lambda$ sufficiently large, and proves the Theorem.
\end{proof}

The existence of solutions of \eqref{eq:hjb2} is a natural adaptation of the elliptic case, whose proof can be found in \cite{crandall2}. Following the ideas of that article, existence of solutions is guaranteed if there exists a continuous subsolution $\underline{u}$ and a continuous supersolution $\overline{u}$ such that (in the case of Dirichlet conditions)
\begin{equation}\label{eq:sottengoppa}
\underline{u}(T,x)=\overline{u}(T,x)=G(x)\,.
\end{equation}
The solution turns to be the following one:
$$
u(t,x)=\sup\{w(t,x)\,:\,\underline{u}\le w\le\overline{u}\hbox{ and $w$ is a subsolution of \eqref{eq:hjb2}}\}\,.
$$
We look for a function $\underline{u}$ of this type:
$$
\underline{u}(t,x)=G(x)+\xi(t),\qquad\mbox{with }\xi(T)=0\,.
$$
Obviously \eqref{eq:sottengoppa} is satisfied. In order to have $\underline{u}$ subsolution of \eqref{eq:hjb2}, it has to be
\begin{equation}\label{eq:zonacesarini}
-\xi'(t)+\lambda\xi(t)\le H^2(t,x,\Delta G(x))+H^1(t,x,\nabla G(x))-\lambda G(x)\,,
\end{equation}
for each $(t,x)\in(0,T)\times\R^d\,.$ If we require $G\in\mathcal{C}^2(\R^d)$ with bounded derivatives, we obtain that the right-hand side of \eqref{eq:zonacesarini} is bounded from below. So it is sufficient to take $\xi$ as the solution of
$$
\begin{cases}
-\xi'(t)+\lambda\xi(t)=-M\,,\\
\xi(T)=0\,,
\end{cases}
$$
with $M$ sufficiently large.

Reasoning in the same way, a good supersolution $\overline{u}$ will be:
$$
\overline{u}(t,x)=G(x)+\xi(t)\,,
$$
with $\xi$ solution, for $M$ sufficiently large, of the ODE
$$
\begin{cases}
-\xi'(t)+\lambda\xi(t)=M\,,\\
\xi(T)=0\,,
\end{cases}
$$
As regards uniqueness, it is an obvious consequence of the comparison principle.

So, we have proved the following
\begin{thm}\label{thm:existence}
Suppose hypotheses of Theorem \ref{thm:verification} and \eqref{eq:bG} are satisfied, and suppose that $G\in\mathcal C^2(\R^d)$ with bounded derivatives.\\
Then there exists a unique viscosity solution for equation \eqref{eq:hjb2}, and, consequently, \eqref{eq:hjb}.
\end{thm}

Now we prove some regularity results for the solution $u$. These results will be essential in order to work with the Fokker-Planck equation.
\begin{thm}\label{thm:Lipsc}
Suppose \eqref{eq:bG} hold true and hypotheses of Theorem \ref{thm:verification} with $\beta=1$ are satisfied. Moreover, suppose that $G$ is a globally Lipschitz function.

Then, every solution of \eqref{eq:hjb2} (and consequently of \eqref{eq:hjb}) is globally Lipschitz in the space variable, with a Lipschitz constant bounded uniformly in $t$.
\end{thm}
\begin{proof}
We have to prove that
$$
|u(t,x)-u(t,y)|\le L|x-y|\,,
$$
for a certain constant $l$. Using Young's inequality, the last inequality is equivalent to
\begin{equation}\label{eq:equivLip}
|u(t,x)-u(t,y)|\le M\left(\delta+\frac{|x-y|^2}{\delta}\right)\,,\qquad\forall\delta>0\,,
\end{equation}
for a certain constant $M>0$. So, we will prove \eqref{eq:equivLip}.

We consider, for $\delta,\gamma,\nu>0$, the following quantity
\begin{equation}\label{eq:doublingL}
u(t,x)-u(t,y)-M\left(\delta+\frac{|x-y|^2}{\delta}\right)-\gamma|x|^2-\frac\nu t\,.
\end{equation}
Due to the coercive term $\gamma|x|^2$ and the boundedness of $u$, we know that the function in \eqref{eq:doublingL} achieves a maximum.

We denote one of its maximum points with $(\bt,\bx,\by)\in(0,T]\times\R^{2d}$.\\
Suppose $\bt=T$. Then
$$
u(t,x)-u(t,y)-M\left(\delta+\frac{|x-y|^2}{\delta}\right)-\gamma|x|^2-\frac\nu t\le G(\bx)-G(\by)-M\left(\delta+\frac{|\bx-\by|^2}{\delta}\right)-\gamma|\bx|^2-\frac\nu T\,,
$$
which implies for $M$ sufficiently large, since $G$ is Lipschitz,
\begin{equation}\label{eq:aimL}
u(t,x)-u(t,y)\le M\left(\delta+\frac{|x-y|^2}{\delta}\right)+\gamma|x|^2+\frac\nu t\,.
\end{equation}
Suppose now $\bt\in(0,T)$. If the maximum achieved by the function in \eqref{eq:doublingL} is non-positive, then \eqref{eq:aimL} remains true.

Now, suppose by contradiction that \eqref{eq:doublingL} attains a strictly positive maximum. This implies
\begin{equation}\label{eq:contrL1}
u(\bt,\bx)-u(\bt,\by)\ge M\left(\delta+\frac{|\bx-\by|^2}{\delta}\right)
\end{equation}
and, for a certain $C>0$,
\begin{equation}\label{eq:contrL2}
\gamma|\bx|^2\le u(\bt,\bx)-u(\bt,\by)\le C\,.
\end{equation}
Applying Proposition \ref{prop:importante} with
\begin{align*}
&u_1(t,x)=u(t,x)\,,\qquad u_2(t,y)=-u(t,y)\,,\\
&\phi(t,x,y)=M\left(\delta+\frac{|x-y|^2}{\delta}\right)+\gamma|x|^2+\frac\nu t\,,
\end{align*}
we obtain that there exist $a,b,X,Y$ such that
\begin{equation*}
(a,\nabla_x\phi(\bt,\bx,\by),X)\in\overline{\mathcal P^{2,+}u(\bt,\bx)}\,,\qquad (-b,-\nabla_y\phi(\bt,\bx,\by),-Y)\in\overline{\mathcal P^{2,-}u(\bt,\by)}\,,
\end{equation*}
and
\begin{equation}\label{eq:comeprima}
a+b=-\frac\nu{\bt^2}\,,\qquad\begin{pmatrix}
X & 0\\
0 & Y
\end{pmatrix}\le D^2\phi(\bt,\bx,\by)\,.
\end{equation}
From now on, we will omit for the function $\phi$ his dependence on $(\bt,\bx,\by)$.

The rest of the proof follows the same ideas of Theorem \ref{thm:comparison}. Since $u$ is both a subsolution and a supersolution, subtracting the two inequalities we obtain
\begin{equation*}
\begin{split}
\frac\nu{\bt^2}+\lambda(u(\bt,\bx)-v(\bt,\by))&\le H^2(\bt,\bx,\mathrm{tr}(X))-H^2(\bt,\by,-\mathrm{tr}(Y))\\
&H^1(\bt,\bx,\nabla_x\phi)-H^1(\bt,\by,-\nabla_y\phi)\,.
\end{split}
\end{equation*}
The left-hand side is estimated with \eqref{eq:contrL1} to obtain
\begin{align*}
\frac\delta 2+\lambda\frac\alpha 2|\bx-\by|^2&\le H^2(\bt,\bx,\mathrm{tr}(X))-H^2(\bt,\by,-\mathrm{tr}(Y))\\
&+H^1(\bt,\bx,\nabla_x\phi)-H^1(\bt,\by,-\nabla_y\phi)\,.
\end{align*}
For the right-hand side term, we compute the derivatives of $\phi$:
\begin{align*}
&\nabla_x\phi=\frac{2M}\delta(\bx-\by)+\frac2\gamma\bx\,,\qquad\nabla_y\phi=-\frac{2M}\delta(\bx-\by)\,,\\
&D^2\phi=\frac{2M}\delta\begin{pmatrix}
I & -I\\
-I & I
\end{pmatrix}+\frac 2\gamma\begin{pmatrix}
I & 0\\
0 & 0
\end{pmatrix}\,.
\end{align*}
We call $\sigma$ the optimal control for $H^2(\bt,\by,-\mathrm{tr}(Y))$. As in \eqref{eq:H2} with $\beta=1$, we obtain
\begin{align*}
H^2(\bt,\bx,\mathrm{tr}(X))-H^2(\bt,\by,-\mathrm{tr}(Y))
\le C|\bx-\by|^\beta+\mathrm{tr}\left(\frac{\Gamma(\bx,\sigma)^2}{2}X+\frac{\Gamma(\by,\sigma)^2}{2}Y\right)\,.
\end{align*}
Defining $B$ as in \eqref{eq:miservedopo}, the last term is estimated in this way:
\begin{align*}
&\mathrm{tr}\left(\frac{\Gamma(\bx,\sigma)^2}{2}X+\frac{\Gamma(\by,\sigma)^2}{2}Y\right)\le\mathrm{tr}(BD^2\phi)\\
=\,&\frac{2M}\delta d|\Gamma(\bx,\sigma)-\Gamma(\by,\sigma)|^2+2\gamma d\Gamma(\bx,\sigma)^2\le CM\left(\delta+\frac{|\bx-\by|^2}\delta\right)+C\gamma\,.
\end{align*}
where $\omega(\alpha)$ is a quantity depending on $\alpha$ such that $\omega(\alpha)\to 0$ when $\alpha\to+\infty$.

In the same way of \eqref{eq:H1}, we obtain the bound for the $H^1$ term:
\begin{align*}
H^1(\bt,\bx,\nabla_x\phi)-H^1(\bt,\by,\nabla_y\phi)\le CM\left(\delta+\frac{|\bx-\by|^2}\delta\right)+C\gamma|\bx|\,.
\end{align*}
Because of \eqref{eq:contrL2}, we have $\gamma|\bx|\to0$ when $\gamma\to 0$. Putting together all the estimates, we obtain, for $M$ large enough,
$$
\lambda M\left(\delta+\frac{|\bx-\by|^2}\delta\right)\le CM\left(\delta+\frac{|\bx-\by|^2}\delta\right)+\omega(\gamma)\,,
$$
where $\omega(\gamma)$ is a quantity that tends to $0$ when $\gamma\to0$.

Hence, for $\gamma$ small enough and $\lambda>C$, we obtain a contradiction.\\
This means that in each case \eqref{eq:aimL} reamins true and, letting $\gamma$ and $\nu$ go to $0$, we get
$$
u(t,x)-u(t,y)\le M\left(\delta+\frac{|x-y|^2}\delta\right)
$$
for all $\delta>0$. Eventually, choosing $\delta=|x-y|$ and reversing the role of $x$ and $y$, we conclude.
\end{proof}
Our next goal is to show that, with a strengthening of hypotheses, $u$ satisfies a stronger estimate, which includes, as a direct consequence, a twice differentiability almost everywhere with respect to the $x$ variable.

For this reason, we introduce the notion of \emph{semiconcavity}:
\begin{defn}
Let $\Omega\subseteq\R^d$. We say that a function $f:\Omega\to\R$ is semiconcave if for all $x\in\Omega$ and for all $h$ such that $x+h$, $x-h\in\Omega$ we have for a certain constant $C$
\begin{equation}\label{eq:semic}
f(x+h)+f(x-h)-2f(x)\le C|h|^2
\end{equation}
\end{defn}
A semiconcave function is twice differentiable almost everywhere. Actually, the following result holds (see \cite{cannarsa}):
\begin{rem}\emph{
The following properties are equivalent:
\begin{itemize}
\item[-] $f$ is a semiconcave function;
\item[-]  There exists a constant $C$ such that $f(x)-C|x|^2$ is a concave function;
\item[-]  $f$ is a.e. twice differentiable and there exists a constant $C$ such that $D^2f\le CI$, where $I$ is the $d\times d$-identity matrix.
\end{itemize}}
\end{rem}
From this remark we immediately obtain:
$$
f\in W^{2,\infty}\implies f\mbox{ semiconcave in }\Omega\,,
$$
since each $W^{2,\infty}$ function satisfies $\norminf{D^2f}\le C$ for a certain $C$, and so the third condition of the remark is satisfied.

More generally, if $f$ is in $W^{2,\infty}$ then $f$ and $-f$ are semiconcave functions.

In order to prove the semiconcave regularity for the value function, we need the following technical lemma on semiconcave and $W^{2,\infty}$ functions:
\begin{lem}
Let $f:\Omega\to\R$. Suppose that the following inequality holds for every $x,y,z\in\Omega$, $\delta>0$ and for a certain $C>0$:
\begin{equation}\label{eq:semic2}
f(x)+f(y)-2f(z)\le C\left(\delta+\frac 1\delta(|x-z|^4+|y-z|^4+|x+y-2z|^2)\right)\,.
\end{equation}
Then $f$ is locally Lipschitz and semiconcave.

Furthermore, if $f$ is Lipschitz and semiconcave then \eqref{eq:semic2} holds for all $x,y,z\in\Omega$, $\delta>0$ and for a certain $C>0$.

Finally, if $f\in W^{2,+\infty}(\Omega)$ then there exists a $C>0$ such that, for every $x,y,z\in\Omega$ and $\delta>0$
\begin{equation}\label{eq:w2inf}
|f(x)+f(y)-2f(z)|\le C\left(\delta +\frac 1\delta(|x-z|^4+|y-z|^4+|x+y-2z|^2)\right).
\end{equation}
\end{lem}
\begin{proof}
If \eqref{eq:semic2} is true, then, taking $x=z+h$, $y=z-h$, $\delta=|h|^2$, we obtain \eqref{eq:semic} and this proves that $f$ is semiconcave.

Furthermore, if we take $z=y$ and $\delta=|x-y|$ from \eqref{eq:semic2} we obtain
$$
f(x)-f(y)\le C\left(1+|x-y|^2\right)|x-y|\,,
$$
proving that $f$ is locally Lipschitz.

On the other hand, suppose that $f$ is Lipschitz and semiconcave. Then we know from \eqref{eq:semic} that
$$
f(\tilde{x}+h)+f(\tilde{x}-h)-2f(\tilde{x})\le C|h|^2\,,
$$
for each $\tilde{x}\in\Omega$, $h\in\R^d$ such that $\tilde{x}+h,\,\tilde{x}-h\in\Omega$.

We choose $\tilde{x}=\frac{x+y}{2}$, $h=\frac{x-y}{2}$, obtaining (up to changing $C$)
$$
f(x)+f(y)-2f\left(\frac{x+y}{2}\right)\le C|x-y|^2\,.
$$
Then we have, using the Lipschitz bound on $f$,
$$
f(x)+f(y)-2f(z)\le C|x-y|^2+2f\left(\frac{x+y}{2}\right)-2f(z)\le C\left(|x-y|^2+|x+y-2z|\right)\,.
$$
Writing $x-y=(x-z)+(z-y)$ and using a generalized Young's inequality, we get
$$
f(x)+f(y)-2f(z)\le C\left(\delta+\frac 1\delta\left(|x-z|^4+|y-z|^4+|x+y-2z|^2\right)\right)\,,
$$
and so \eqref{eq:semic2}. Finally, if $f$ is $W^{2,\infty}$ then $f$ and $-f$ are Lipschitz and semiconcave functions. Then \eqref{eq:semic2} holds for $f$ and $-f$, proving that \eqref{eq:w2inf} is true.
\end{proof}
\begin{prop}
Suppose hypotheses of Theorem \ref{thm:Lipsc} are satisfied. Futhermore, suppose that $G$, $L^1$ and $L^2$ are Lipschitz and semiconcave functions in the space variable, and $b$ and $\Gamma$ are $W^{2,\infty}$ functions in the space variable.

Then, every solution of \eqref{eq:hjb2} is semiconcave in the space variable.
\end{prop}
\begin{proof}
The proof readapts the ideas exploited in \cite{crandall2,crandall,ricciardi1}. We want to prove that
$$
u(t,x)+u(t,y)-2u(t,z)\le C\left(\delta+\frac 1\delta\left(|x-z|^4+|y-z|^4+|x+y-2z|^2\right)\right)\,,
$$
for a certain $C>0$ and for all $\delta>0$, $t\in(0,T]$, $x,y,z\in\R^d$.

To do that, we argue as in the Lipschitz case, following the ideas of \emph{Theorem VII.3} of \cite{ishii}. We consider the following auxiliary function:
\begin{equation}\label{eq:phiseco}
\phi(t,x,y,z)=M\seco+\gamma|x|^2+\frac\nu t\,.
\end{equation}
Due to the coercive term $\gamma|x|^2$ and the boundedness of $u$, we know that the quantity
\begin{equation}\label{eq:tripling}
u(t,x)+u(t,y)-2u(t,z)-\phi(t,x,y,z)
\end{equation}
achieves a maximum in a certain point, say $(\bt,\by,\bx,\bz)\in(0,T]\in\R^{3n}\,.$

Suppose $\bt=T$. Then we have
$$
u(t,x)+u(t,y)-2u(t,z)-\phi(t,x,y,z)\le G(\bx)+G(\by)-2G(\bz)-\phi(T,\bx,\by,\bz)\,.
$$
Since $G$ is Lipschitz and semiconcave, the right-hand side term is non-positive for $M$ sufficiently large, using \eqref{eq:semic2}. This implies
\begin{equation}\label{eq:quellochevogliamo}
u(t,x)+u(t,y)-2u(t,z)\le\phi(t,x,y,z)\,.
\end{equation}
Suppose now $\bt\in(0,T)$. If the maximum of \eqref{eq:tripling} is non-positive, then \eqref{eq:quellochevogliamo} remains true.

Now, suppose by contradiction that \eqref{eq:tripling} attains a strictly positive maximum. This implies
\begin{equation}\label{eq:nauta}
u(\bt,\bx)+u(\bt,\by)-2u(\bt,\bz)\ge M\seca
\end{equation}
and, since $u$ is Lipschitz, for a certain $C>0$ 
\begin{equation}\label{eq:elparty}
\gamma|\bx|^2\le\usb\le C\left(|\bx-\bz|+|\by-\bz|\right)\,.
\end{equation}
Since $-2u(t,x)$ is a subsolution of
$$
-z_t-\tilde{H}^2(t,x,\Delta z)-\tilde{H}^1(t,x,\nabla z)=-2F(t,x)\,,
$$
with $\tilde{H}^2(t,x,q)=-2H^2(t,x,-\miezz q)$ and $\tilde{H}^1(t,x,p)=-2H^1(t,x,-\miezz p)$, we can apply Proposition \ref{prop:importante} with
\begin{align*}
&u_1(t,x)=u(t,x)\,,\qquad u_2(t,y)=u(t,y)\,,\qquad u_3(t,z)=-2u(t,z)\,,\\
&\phi(t,x,y,z)\mbox{ as defined before},
\end{align*}
obtaining that there exist $a,b,c,X,Y,Z$ such that
\begin{align*}
&(a,\nabla_x\phi(\bt,\bx,\by,\bz),X)\in\overline{\mathcal P^{2,+}u(\bt,\bx)}\,,\qquad (b,\nabla_y\phi(\bt,\bx,\by,\bz),Y)\in\overline{\mathcal P^{2,+}u(\bt,\by)}\,,\\
&(c,\nabla_z\phi(\bt,\bx,\by,\bz),Z)\in\overline{\mathcal P^{2,+}(-2u)(\bt,\bz)}\,,
\end{align*}
and
\begin{equation}\label{eq:tuttlabel}
a+b+c=-\frac{\nu}{\bt^2}\,,\qquad\begin{pmatrix}
X & 0 & 0\\
0 & Y & 0\\
0 & 0 & Z
\end{pmatrix}\le D^2\phi(\bt,\bx,\by,\bz)\,.
\end{equation}
From now on, we will omit for the function $\phi$ its dependence on $(\bt,\bx,\by,\bz)$.

It is immediate to prove that, if $(c,\nabla_z\phi(\bt,\bx,\by,\bz),Z)\in\overline{\mathcal{P}^{2,+}(-2u)(\bt,\bz)}$, then
$$
\left(-\frac c2,\,-\miezz\nabla_z\phi(\bt,\bx,\by,\bz),\,-\miezz Z\right)\in\overline{\mathcal P^{2,-}u(\bt,\bz)}\,.
$$
Hence, since $u$ is both a subsolution and a supersolution, one has
\begin{align*}
-&a-H^2(\bt,\bx,\mathrm{tr}(X))-H^1(\bt,\bx,\nabla_x\phi)+\lambda u(\bt,\bx)\le0\,,\\
-&b-H^2(\bt,\by,\mathrm{tr}(Y))-H^1(\bt,\by,\nabla_y\phi)+\lambda u(\bt,\by)\le0\,,\\
&\frac c2-H^2\left(\bt,\bz,-\miezz\mathrm{tr}(Z)\right)-H^1\left(\bt,\bz,-\miezz\nabla_z\phi\right)+\lambda u(\bt,\bz)\ge0\,.
\end{align*}
Adding the first two inequalities and subtracting twice the third we obtain
\begin{align*}
&\frac\nu{\bt^2}+\lambda(\usb)\\
\le& H^2(\bt,\bx,\mathrm{tr}(X))+H^2(\bt,\by,\mathrm{tr}(Y))-2H^2\left(\bt,\bz,-\miezz\mathrm{tr}(Z)\right)\\
+&H^1(\bt,\bx,\nabla_x\phi)+H^1(\bt,\by,\nabla_y\phi)-2H^1\left(\bt,\bz,-\miezz\nabla_z\phi\right)\,.
\end{align*}
We have $\frac\nu{\bt^2}\ge0$. As regards the second term, with \eqref{eq:nauta} we get
\begin{equation}\label{eq:checugliazza}
\begin{split}
&\lambda M\seca\\
\le
&H^2(\bt,\bx,\mathrm{tr}(X))+H^2(\bt,\by,\mathrm{tr}(Y))-2H^2\left(\bt,\bz,-\miezz\mathrm{tr}(Z)\right)\\
+ &H^1(\bt,\bx,\nabla_x\phi)+H^1(\bt,\by,\nabla_y\phi)-2H^1\left(\bt,\bz,-\miezz\nabla_z\phi\right)\,.
\end{split}
\end{equation}
Computing the derivatives of $\phi$, one has
\begin{align*}
&\nabla_x\phi=\frac{4M}{\delta}\left(|\bx-\bz|^2(\bx-\bz)+(\bx+\by-2\bz)\right)+2\gamma\bx\,,\\
&\nabla_y\phi=\frac{4M}{\delta}\left(|\by-\bz|^2(\by-\bz)+(\bx+\by-2\bz)\right)\,,\\
&\nabla_z\phi=-\frac{4M}{\delta}\left(|\bx-\bz|^2(\bx-\bz)+|\by-\bz|^2(\by-\bz)+2(\bx+\by-2\bz)\right)\,,
\end{align*}
and so
\begin{align*}
D^2_{xx}\phi=\frac{4M}{\delta}\left(2(\bx-\bz)\otimes(\bx-\bz)+|\bx-\bz|^2I+I\right)+2\gamma I\,,\qquad D^2_{xy}\phi=\frac{4M}{\delta}I\,,\\
D^2_{yy}\phi=\frac{4M}{\delta}\left(2(\by-\bz)\otimes(\by-\bz)+|\by-\bz|^2I+I\right)\,,\\
D^2_{xz}\phi=-\frac{4M}{\delta}\left(2(\bx-\bz)\otimes(\bx-\bz)+|\bx-\bz|^2I+2I\right)\,,\\
D^2_{yz}\phi=-\frac{4M}{\delta}\left(2(\by-\bz)\otimes(\by-\bz)+|\by-\bz|^2I+2I\right)\,,\\
D^2_{zz}\phi=\frac{4M}{\delta}\left(2(\bx-\bz)\otimes(\bx-\bz)+|\bx-\bz|^2I+2(\by-\bz)\otimes(\by-\bz)+|\by-\bz|^2I+4I\right)\,.
\end{align*}
We estimate the $H^2$ terms in \eqref{eq:checugliazza}. Choosing $\sigma$ as the optimal control for $H^2(\bt,\bz,-\miezz\mathrm{tr}(Z))$, we obtain
\begin{align*}
H^2(\bt,\bx,\mathrm{tr}(X))&+H^2(\bt,\by,\mathrm{tr}(Y))-2H^2\left(\bt,\bz,-\miezz\mathrm{tr}(Z)\right)\le\frac{\Gamma(\bx,\sigma)^2}2\mathrm{tr}(X)+\frac{\Gamma(\by,\sigma)^2}{2}\mathrm{tr}(Y)\\
&+\frac{\Gamma(\bz,\sigma)^2}{2}\mathrm{tr}(Z)+e^{-\lambda(T-t)}\left(L^2(\bx,\bx,\sigma)+L^2(\bt,\by,\sigma)-2L^2(\bt,\bz,\sigma)\right)\\
&\le C\phi+\miezz\mathrm{tr}\Big(\Gamma(\bx,\sigma)^2X+\Gamma(\by,\sigma)^2Y+\Gamma(\bz,\sigma)^2Z\Big).
\end{align*}
We estimate the trace term in the following way:
\begin{align*}
&\mathrm{tr}\Big(\Gamma(\bx,\sigma)^2X+\Gamma(\by,\sigma)^2Y+\Gamma(\bz,\sigma)^2Z\Big)\\
=\,&\mathrm{tr}\left(\begin{pmatrix}
\Gx^2I & \Gx\Gy I & \Gx\Gz I\\
\Gx\Gy I& \Gy^2 I & \Gy\Gz I\\
\Gx\Gz I & \Gy\Gz I & \Gz^2 I
\end{pmatrix}\begin{pmatrix}
X & 0 & 0\\
0 & Y & 0\\
0 & 0 & Z
\end{pmatrix}\right)\\
\le\,&\mathrm{tr}\left(\begin{pmatrix}
\Gx^2I & \Gx\Gy I & \Gx\Gz I\\
\Gx\Gy I& \Gy^2 I & \Gy\Gz I\\
\Gx\Gz I & \Gy\Gz I & \Gz^2 I
\end{pmatrix}D^2\phi\right)\,,
\end{align*}
since the matrix on the left is non-negative definite and thanks to \eqref{eq:tuttlabel}.

So we obtain, with standard computations,
\begin{align*}
&\mathrm{tr}(\Gx^2 X + \Gy^2 Y+\Gz^2 Z)\\
\le\,&\frac{4M}{\delta}\big[(2+d)|\bx-\bz|^2(\Gx-\Gz)^2+(2+d)|\by-\bz|^2(\Gy-\Gz)^2\\
+\,&d(\Gx-\Gy-\Gz)^2\big]+2\gamma d\,\Gx^2\,.
\end{align*}
Using the hypotheses on $\Gamma$, we finally obtain this bound for the $H^2$ term:
$$
H^2(\bt,\bx,\mathrm{tr}(X))+H^2(\bt,\by,\mathrm{tr}(Y))-2H^2\left(\bt,\bz,-\miezz\mathrm{tr}(Z)\right)\le C\phi+C\gamma\,
$$
where $C$ depends also on $d$.

We argue in a similar way in order to bound the $H^1$ term. We have, choosing $\alpha$ as the optimal control for $H^1(\bt,\bx-\miezz\nabla_z\phi)$, and since $\nabla_z\phi=-(\nabla_x\phi+\nabla_y\phi)+2\gamma\bx\,,$
\begin{align*}
&H^1(\bt,\bx,\nabla_x\phi)+H^1(\bt,\by,\nabla_y\phi)-2H^1\left(\bt,\bz,-\miezz\nabla_z\phi\right)\le\langle b(\bx,\alpha),\nabla_x\phi\rangle+\langle b(\by,\alpha),\nabla_y\phi\rangle\\
+\,&\langle b(\bz,\alpha),\nabla_z\phi\rangle+e^{-\lambda(T-t)}\left(L^1(\bt,\bx,\alpha)+L^1(\bt,\by,\alpha)-2L^1(\bt,\bz,\alpha)\right)\\
\le\,&C\phi+\frac{4M}\delta|\bx-\bz|^2\langle b(\bx,\alpha)-b(\bz,\alpha),\bx-\bz\rangle +\frac{4M}\delta|\by-\bz|^2\langle b(\by,\alpha)-b(\bz,\alpha),\by-\bz\rangle\\
+\,&\frac{2M}{\delta}\langle b(\bx,\alpha)+b(\by,\alpha)-2b(\bz,\alpha),\bx+\by-2\bz\rangle+2\gamma\langle b(\bx,\alpha),\bx\rangle\le C\phi+C\gamma|\bx|^2\,.
\end{align*}
Because of \eqref{eq:elparty}, we have $\gamma|bx|\to0$ when $\gamma\to0$.

Putting together all the estimates, we obtain, for $M$ large enough,
$$
\lambda\phi\le C\phi+\omega(\gamma)\,,
$$
where as always $\omega(\gamma)$ is a quantity converging to $0$ when $\gamma\to0$.

Hence, for $\gamma$ small enough and $\lambda\gg1$, we obtain a contradiction.

So in each case \eqref{eq:quellochevogliamo} remains true and, letting $\gamma$ and $\nu$ go to $0$, we get
$$
u(t,x)+u(t,y)-2u(t,z)\le M\seco\,,
$$
which concludes the proof.
\end{proof}

\section{Classical solutions in a regular case}\label{sec4}
In this section we want to prove that, for a specified class of $\mathcal{C}^2$ Hamiltonians, the solution $u$ of \eqref{eq:hjb} is actually in the space $\mathcal{C}^{1+\frac\gamma 2,2+\gamma}$, for some $\gamma>0$. This estimate allows us to linearize the problem \eqref{eq:hjb} and apply the classical regularity results of linear parabolic equations in order to obtain a $\mathcal{C}^{1+\frac\alpha 2,3+\alpha}$, depending only on the coefficients $H^1$, $H^2$ and the data $G$. This last estimation will be crucial in order to obtain a $\mathcal{C}^{1+\frac\alpha 2,3+\alpha}$ solution in our framework.

To do that, we need to use the regularity results obtained by Krylov in 1983. According to \cite{krylov1}, we define a class of functions for which the $\mathcal{C}^{1+\frac\gamma 2,2+\gamma}$ regularity will hold.
\begin{defn}\label{def:krylov}
Consider a function $M:[0,T]\times\R\times\R_{>0}\times\R^{d^2}\times\R^d\times\R\to\R$ of variables $(t,x,\beta,B,\underline{p},s)$, with $B={(b_{ij})}_{ij}$ and $\underline{p}={(p_i)}_i$.

We say that $M\in\mathscr M$ if the following conditions are satisfied:
\begin{itemize}
\item [(i)] $M$ is positive homogeneous of first order, with respect to the variables $(\beta,B,\underline{p},s)\,$;
\item [(ii)] $M$ is twice continuously differentiable of first order, with respect to the variables $(\beta,B,\underline{p},s)\,$;
\item [(iii)] $\exists\,\nu>0$ such that
$$
\sum\limits_{i,j}\frac{\partial M}{\partial b_{ij}}\xi_i\xi_j\ge\nu|\xi|^2\,,\qquad\forall\,\xi\in\R^d\,;
$$
\item [(iv)] $M$ is concave with respect to ${(b_{ij})}_{ij}\,$;
\item [(v)] The second order directional derivative of $M$ with respect to $(B,\underline{p},s)$ along a vector $(B_0,\underline{p_0},s_0)$ is bounded from above by $C\beta^{-1}\big(|\underline{p_0}|^2+|s_0|^2\big)\,,$;
\item [(vi)] There exists $C>0$ such that, for all $i$ and $j\,$,
\begin{equation}\label{eq:hpkr}
|M_{b_{ij}}|+|M_{p_i}|+|M_\beta|+|M_{b_{ij}x}|+|M_{sx}|+|M_{\beta x}|\le C\,;
\end{equation}
\item [(vii)] M continuous in all variables and differentiable with respect to $t$, with
$$
|M_t(t,x,\beta,B,\underline{p},s)|+|M_{x_ix_j}(t,x,\beta,B,\underline{p},s)|\le C\sqrt{\beta^2+s^2+|\underline{p}|^2+|B|^2}\,.
$$
\end{itemize}
\end{defn}
With these hypotheses, we can state the main result we will use in order to prove the classical regularity of $u$. This is \emph{Theorem 1.1} of \cite{krylov1}.
\begin{thm}\label{thm:krylov}
Let, for $r\in\N$, $\{(M_r)\}_r$ a sequence of functions such that $M_r\in\mathscr M$ $\forall r$. Let moreover $\phi\in\mathcal C^{2+\alpha}$, for a certain $0<\alpha<1$. Then, if we define $M=\inf\limits_r M_r$, we have that the problem
$$
\left\{
\begin{array}{lr}
u_t-M(t,x,1,D^2u,Du,u)=0\,,\qquad& (t,x)\in(0,T)\times\R^d\,,\\
u(0,x)=\phi(x)\,, & x\in\R^d\,,
\end{array}
\right.
$$
admits a unique solution $u\in\mathcal C^{2+\gamma}([0,T]\times\R^d)$, where $\gamma\in(0,1)\,$.
\end{thm}
Now we are ready to prove that our problem \eqref{eq:hjb} admits a classical solution.
\begin{thm}\label{thm:appliedkrylov}
Let $H^1$ and $H^2$ be differentiable with respect to $t$ and twice continuously differentiable with respect to the other variables. Let, moreover, $G\in\mathcal C^{2+\alpha}(\R^d)$, for a certain $0<\alpha<1$.

Furthermore, we suppose that $H^2$ and $H^1$ are concave in the last variable, and
\begin{equation}\label{eq:hp0}
H^2_q(t,x,q)\ge\nu\qquad\forall(t,x,q)\in[0,T]\times\R^d\times\R\mbox{ and for a certain }\nu>0\,.
\end{equation}
Finally, we require the following assumptions: $\exists\,C>0$ such that
\begin{gather}
\norminf{H^1(\cdot,\cdot,0)}+\norminf{H^2(\cdot,\cdot,0)}\le C\,,\label{eq:hp1}\\
\norminf{H^1_p(\cdot,\cdot,\cdot)}+\norminf{H^2_q(\cdot,\cdot,\cdot)}\le C\,,\label{eq:hp2}\\
\norminf{H^2_{qx}(\cdot,\cdot,\cdot)}\le C\,,\label{eq:hp3}\\
H^1_p(t,x,p)\cdot p-H^1(t,x,p)\ge-C\,,\qquad H^2_q(t,x,q)\cdot q-H^2(t,x,q)\ge-C\,,\label{eq:hp4}\\
|H^1_{px_i}(t,x,p)\cdot p-H^1_{x_i}(t,x,p)|\le C\,,\qquad |H^2_{qx_i}(t,x,q)\cdot q-H^2_{x_i}(t,x,q)|\le C\,,\label{eq:hp5}\\
\norminf{H^1_{xx}(\cdot,\cdot,p)}+\norminf{H^1_t(\cdot,\cdot,p)}\le C(1+|p|)\,,\label{eq:hp6}\\
\norminf{H^2_{xx}(\cdot,\cdot,q)}+\norminf{H^2_t(\cdot,\cdot,q)}\le C(1+|q|)\,.\label{eq:hp7}
\end{gather}
Then, if $u$ is the unique solution of \eqref{eq:hjb}, we have $u\in\mathcal C^{1+\frac\gamma 2,2+\gamma}([0,T]\times\R^d)$, for a certain $0<\gamma<1$.
\end{thm}
\begin{proof}
We define $\forall r$ the function $M_r=M$, where
$$
M(t,x,\beta,B,\underline{p},s)=\beta H^2(t,x,\beta^{-1}\mathrm{tr}(B))+\beta H^1(t,x,\beta^{-1}\underline{p})\,.
$$
We note that this definition is well-posed, since $\beta\in R>0$, and we have
$$
M(t,x,1,D^2 u,Du, u)=H^2(t,x,\Delta u)+H^1(t,x,Du)\,.
$$
So, if we can apply Theorem \ref{thm:krylov}, we will obtain a $\mathcal C^{1+\frac\gamma 2,2+\gamma}$ for \eqref{eq:hjb}, and we will have finished.

We only have to check that the conditions in Definition \ref{def:krylov} are satisfied by $M$.
\begin{itemize}
\item [\emph{(i), (ii)}] The positive homogeneity with respect to $(\beta, B,\underline{p})$ and the twice continuously differentiability with respect to $(x,\beta,B,\underline{p})$ are immediate consequences of the definition of $M$ and the hypotheses of regularity for $H^1$, $H^2$;\\
\item [\emph{(iii)}] Since $M$ depends on $b_{ij}$ only when $i=j$, we have $\forall\xi\in\R^d$
$$
\sum\limits_{i,j}\frac{\partial M}{\partial b_{ij}}\xi_i\xi_j=\sum\limits_i\beta\beta^{-1}H^2_q(t,x,\beta^{-1}\mathrm{tr}(B))\xi_i^2\ge\nu|\xi|^2\,;
$$
\item [\emph{(iv)}] Since $\beta>0$, the concavity of $M$ with respect to $(b_{ij})$ is a direct consequence of the concavity of $H^2$ with respect to $q$. Actually, we have
$$
\frac{\partial^2 M}{\partial b_{ij}^2}=\beta\beta^{-2}H^2_{qq}(t,x,\beta^{-1}\mathrm{tr}(B))\le 0\,;
$$
\item [\emph{(v)}] Take a vector $(B_0,\underline{p_0})$ and define the function $\phi(r)=M(t,x,\beta,B+rB_0,\underline{p}+r\underline{p_0}).$ We have to prove that
$$
\phi''(0)\le C\beta^{-1}|\underline{p_0}|^2\,.
$$
A straightforward computation shows that
$$
\phi''(0)=\beta^{-1}\left[H^2_{qq}(t,x,\beta^{-1}\mathrm{tr}(B))\mathrm{tr}(B_0)^2+\langle H^1_{pp}(t,x,\beta^{-1}\underline{p})\underline{p_0},\underline{p_0}\rangle\right]\,.
$$
Using the concavity of $H^2$ and the semiconcavity of $H^1$ with respect to the last variable, we are done;
\item [\emph{(vi)}] We analyze each term of \eqref{eq:hpkr}. As regards $M_{b_{ij}}$ and $M_{p_i}$, we have for \eqref{eq:hp2}
$$
|M_{b_{ij}}|+|M_{p_i}|=|H^2_q(t,x,\beta^{-1}\mathrm{tr}(B))|+|H^1_p(t,x,\beta^{-1}\underline{p})|\le C\,.
$$
For $M_\beta$, we compute the derivative and we obtain
\begin{align*}
M_\beta &= H^2(t,x,\beta^{-1}\mathrm{tr}(B))-\beta^{-1}\mathrm{tr}(B)H^2_q(t,x,\beta^{-1}\mathrm{tr}(B))\\
&+ H^1(t,x,\beta^{-1}\underline{p})-\beta^{-1}\underline{p}H^1_p(t,x,\beta^{-1}\underline{p})\,.
\end{align*}
Hence, the concavity of $H^2$ and $H^1$, with conditions \eqref{eq:hp1} and \eqref{eq:hp4}, easily allow to obtain a bound for $|M_\beta|$. Similar and easier computations are made in order to bound $|M_{b_{ij}x}|$ and $|M_{\beta x}|$, using \eqref{eq:hp3} and \eqref{eq:hp5}.
\item [\emph{(vii)}] The continuity of $M$ follows from the continuity of $H^1$ and $H^2$. We prove only the estimate for $|M_{x_ix_j}|$, since the $|M_t|$ goes along the same computations. We have
\begin{align*}
|M_{x_i x_j}| &\le |\beta H^2_{x_ix_j}(t,x,\beta^{-1}\mathrm{tr}(B))|+|\beta H^1_{x_i x_j}(t,x,\beta^{-1}\underline{p})|\\
&\le C\big(|\mathrm{tr}(B)|+|\underline{p}|+\beta\big)\le C\sqrt{\beta^2+|\underline{p}|^2+\norm{B}^2}\,,
\end{align*}
where we used \eqref{eq:hp6}, \eqref{eq:hp7}.
\end{itemize}
Then $M\in\mathscr M$ and we are allowed to apply Theorem \ref{thm:krylov}. This concludes the proof.
\end{proof}
Since $u$ is a classical solution, we are allowed to linearize problem \eqref{eq:hjb}. Actually, we get that $u$ satisfies
\begin{equation}\label{eq:linhjb}
\begin{cases}
-u_t-V(t,x)\Delta u-Z(t,x) Du=b(t,x)\,,\\
u(T,x)=G(x)\,,
\end{cases}
\end{equation}
where
\begin{align*}
&V(t,x)=\int_0^1 H^2_q(t,x,\lambda\Delta u(t,x))\,d\lambda\,,\qquad\quad Z(t,x)=\int_0^1 H^1_p(t,x,\lambda Du(t,x))\,d\lambda\,,\\
&b(t,x)=H^2(t,x,0)+H^1(t,x,0)\,.
\end{align*}
Now we can use the linear character of the equation \eqref{eq:linhjb} in order to improve the regularity of $u$. The Theorem is the following:
\begin{thm}
Suppose the hypotheses of Theorem \ref{thm:appliedkrylov} are satisfied. Furthermore, assume that, for some $k>0$ and $C>0$,
\begin{equation}\label{eq:hp12}
|H^2_x(t,x,q)|+|H^1_x(t,x,p)|\le C(1+|p|)\,,\qquad\forall t\in[0,T]\,,\,\,\forall x\in\R^d\,.
\end{equation}
Moreover, for a certain $0<\alpha<1$, we require that $\forall L>0$ $\exists C_L>0$ such that $\forall |p|\le L$, $|q|\le L$ it holds
\begin{equation}\label{eq:hp13}
\norm{H^2_q(\cdot,\cdot,q)}_{\frac\alpha 2,\alpha}+\norm{H^1_p(\cdot,\cdot,p)}_{\frac\alpha 2,\alpha}+\norm{H^2_x(\cdot,\cdot,q)}_{\frac\alpha 2,\alpha}+\norm{H^1_x(\cdot,\cdot,p)}_{\frac\alpha 2,\alpha}\le C_L\,,
\end{equation}
and all the derivatives $H^2_q$, $H^1_p$, $H^2_x$, $H^1_x$ are Lipschitz in the last variable.

Finally, suppose that $G\in\mathcal C^{3+\alpha}$, with
\begin{equation}\label{eq:hp14}
\norm{G}_{3+\alpha}\le C\,.
\end{equation}
Then, the solution $u$ of \eqref{eq:hjb} satisfies the following estimate:
\begin{equation}\label{eq:4.15}
\norm{u}_{1+\frac\alpha 2,3+\alpha}\le C\,,
\end{equation}
where $C$ depends on $H^2$, $H^1$, $G$, $T$, $d$.
\end{thm}
\begin{proof}
We consider, for $1\le i\le d$, the function $v:=u_{x_i}$.
Differentiating the equation \eqref{eq:hjb} with respect to $x_i$, we obtain for the function $v$
\begin{equation}\label{eq:4.16}
\begin{cases}
-v_t-H^2_q(t,x,\Delta u)\Delta v-H^1_p(t,x,Du)\cdot Dv=H^2_{x_i}(t,x,\Delta u)+H^1_{x_i}(t,x,Du)\,,\\
v(T,x)=G_{x_i}(x)\,.
\end{cases}
\end{equation}
Thanks to the hypotheses, the coefficients $H^2_q$, $H^1_p$ are in $L^\infty$ and $G_{x_i}\in\mathcal C^{2+\alpha}$. Moreover, with the same ideas of Theorem \ref{thm:Lipsc} (see also \cite{PorrettaPriola}), we have that $|Du|$ is globally bounded in $L^\infty$, which implies, thanks to \eqref{eq:hp12},
$$
H^2_{x_i}(t,x,\Delta u)+H^1_{x_i}(t,x,Du)\in L^\infty.
$$
With these hypotheses we know from standard regularity results (see e.g. \cite{lsu}) that $\norm{v}_{\frac{1+\alpha}{2},1+\alpha}\le C\,$, where $C$ depends on $G$, $H^2$, $H^1$.

Coming back to the linear equation \eqref{eq:linhjb} satisfied by $u$, we get that
$$
V,Z,b\in\mathcal{C}^{\frac\alpha 2,\alpha},\qquad G\in\mathcal C^{2+\alpha}\,.
$$
Hence, \emph{Theorem IV.5.1} of \cite{lsu} implies that $u\in\mathcal C^{1+\frac\alpha 2,2+\alpha}$ and $\norm{u}_{1+\frac\alpha 2,2+\alpha}\le C$.

Now, the coefficients of the equation \eqref{eq:4.16} are in $\mathcal C^{\frac\alpha 2,\alpha}$ and $G_{x_i}\in\mathcal{C}^{2+\alpha}.$ Applying again \emph{Theorem IV.5.1} of \cite{lsu} we obtain
$$
\norm{v}_{1+\frac\alpha 2,2+\alpha}\le C\implies \norm{u}_{1+\frac\alpha 2,3+\alpha}\le C\,,
$$
which concludes the proof.
\end{proof}

\section{Regular solutions for the Bellman operator}\label{sec5}
It is not obvious at all to have in our examples $\mathcal C^2$ Hamiltonian in all variables. In order to handle this problem, we introduce the following Proposition.
\begin{prop}\label{prop:prop}
Let $H^2$ and $H^1$ be almost everywhere twice differentiable in the space and in the last variable, and almost everywhere differentiable in the time variable, and suppose $H^1$ and $H^2$ be concave in the last variable and satisfying estimates from \eqref{eq:hp0} to \eqref{eq:hp7} and from \eqref{eq:hp12} and \eqref{eq:hp14}, where pointwise estimates have to be intended almost everywhere.

Then the unique solution $u$ of \eqref{eq:hjb} satisfies
\begin{equation}\label{eq:5.1}
\norm{u}_{1+\frac\alpha 2,3+\alpha}\le C\,.
\end{equation}
\end{prop}
\begin{proof}
We consider two sequences $H^{2,k}$ and $H^{1,k}$ of smooth functions converging to $H^2$ and $H^1$ taking, for $\delta>0$, the convolutions
$$
H^{2,k}=H^2\star\rho_\delta\,,\qquad H^{1,k}=H^1\star\rho_\delta\,,
$$
where $\rho_\delta\ge0$ is a non-negative function with compact support in $B(0,\delta)$. Then we take the related solutions $u^k$ of \eqref{eq:hjb}.

It is immediate to prove that $H^{2,k}$ and $H^{1,k}$ satisfy conditions from \eqref{eq:hp0} to \eqref{eq:hp7} and from \eqref{eq:hp12} to \eqref{eq:hp14}, with all bounds independent on $k$. Just to give an example, we show that
$$
H^{1,k}_p(t,x,p)\cdot p-H^{1,k}(t,x,p)\ge -C\,.
$$
We have
\begin{align*}
&\qquad\qquad\qquad\qquad\qquad\qquad\quad H_p^{1,k}(t,x,p)\cdot p-H^{1,k}(t,x,p)\\
&=\int_0^T\int_{\R^d}\int_{-\infty}^{+\infty}\Big(H_p^1(t-s,x-y,p-r)\cdot p-H^1(t-s,x-y,p-r)\Big)\rho_\delta(s,y,r)\,dsdydr\\
&=\int_0^T\int_{\R^d}\int_{-\infty}^{+\infty}\Big(H_p^1(t-s,x-y,p-r)\cdot (p-r)-H^1(t-s,x-y,p-r)\Big)\rho_\delta(s,y,r)\,dsdydr\\
&+\int_0^T\int_{\R^d}\int_{-\infty}^{+\infty}\Big(H_p^1(t-s,x-y,p-r)\cdot r\Big)\rho_\delta(s,y,r)\,dsdydr\ge-C\,,
\end{align*}
where we use in the last passage hypotheses \eqref{eq:hp2} and \eqref{eq:hp4} for $H^1$.

Then, from the results of the previous  sections,
$$
\norm{u_k}_{1+\frac\alpha 2,3+\alpha}\le C\,.
$$
Hence, we can use Ascoli-Arzelà Theorem on any compact set $K\subset\R^d$ and obtain that $\exists\, u$ such that, up to subsequences, $u^k\to u$ pointwise with all the derivatives. Moreover, $u$ satisfies \eqref{eq:5.1}.

Passing to the limit in the equation of $u^k$, we obtain that $u$ satisfies \eqref{eq:hjb}. This concludes the Proposition.
\end{proof}

\begin{rem}
We stress the fact that condition \eqref{eq:hp4} is satisfied at least in the model case \eqref{eq:L3}, where
\begin{equation*}
\begin{split}
H^1(t,x,p)&=\inf\limits_{\alpha\in\mathcal U}\big\{\langle p,\alpha\rangle + L^1(t,x,\alpha)\big\}\,,\\
H^2(t,x,q)&=\inf\limits_{\eta\in\mathcal{S'}}\big\{\eta q+L_3(t,x,\eta)\big\}\,,
\end{split}
\end{equation*}
provided $L^1$ and $L_3$ are uniformly bounded and strictly convex with respect to the last variable.

Actually, the linear character of $\langle p,\alpha\rangle$ and $\eta q$ with respect to $\alpha$ and $\eta$ and the strict convexity of the Lagrangian functions immediately imply that the $\inf$ in $H^1$ and $H^2$ is attained at a unique point $\alpha_{t,x,p}$, $\eta_{t,x,q}$.

The uniqueness of the $argmin$ plays a crucial role in order to obtain the $\mathcal C^1$ character of $H^1$ and $H^2$ with respect to the last variable. Actually, we have
$$
H_p^1(t,x,p)=\alpha_{t,x,p}\,,\qquad H^2_q(t,x,q)=\eta_{t,x,q}\,,
$$
and so
\begin{align*}
|H_p^1(t,x,p)\cdot p-H^1(t,x,p)|=|L^1(t,x,\alpha_{t,x,p})|\le C\,,\\
|H^2_q(t,x,q)\cdot \-H^2(t,x,q)|=|L_3(t,x,\eta_{t,x,p})|\le C\,.
\end{align*}
Moreover, if we strengthen the regularity hypotheses on $L^1$ and $L_3$ with respect to all variables, we obtain that also conditions \eqref{eq:hp0},\dots,\eqref{eq:hp7} of Theorem \ref{thm:appliedkrylov} are satisfied.

Eventually, we note that the convexity condition of $L_3$ is satisfied if we require, in \eqref{eq:L2}, that $L^2$ is convex and non-increasing with respect to the last variable. Actually, from Remark \ref{rem},
$$
L_3(t,x,\eta)=L^2(t,x,\sqrt 2\eta)\,,
$$
and so, if $L^2$ is twice differentiable in the last variable, so is $L_3$, and it holds
$$
\partial^2_{\eta \eta}L_3(t,x,\eta)=\partial^2_{\sigma\sigma} L^2(t,x,\sqrt{2\eta})(2\eta)^{-\miezz}-\partial_\sigma L^2(t,x,\sqrt{2\eta})(2\eta)^{-\frac 32}>0\,.
$$
\end{rem}

\section{The Fokker-Planck Equation and The Mean Field Games System}\label{sec6}
The existence and uniqueness for the Fokker-Planck equation follow from well-known arguments, thanks to its linearity character. We just report the definition of solutions we will consider throughout the rest of the paper.

\begin{defn}\label{def:FP}
Let $m\in\mathcal C([0,T];\mathcal P(\R^d))$, $a\in L^\infty([0,T]\times\R^d)$ and $b\in L^\infty([0,T]\times\R^d;\R^d)$. Moreover, let $m_0\in\mathcal P(\R^d)$. Then we say that $m$ solves (in the sense of distribution) the Fokker-Planck equation
\begin{equation}\label{eq:fp}
\begin{split}
\left\{
\begin{array}{ll}
m_t-\Delta(a(t,x)m)+\mathrm{div}(mb(t,x))=0\,, & (t,x)\in(0,T)\times\R^d\,,\\
m(0,x)=m_0(x)\,, & x\in\R^d\,,
\end{array}
\right.
\end{split}
\end{equation}
if, for all smooth $\phi,\psi\in\mathcal C^\infty([0,T]\times\R^d)$, $\phi_T\in\mathcal C^\infty(\R^d)$, such that
$$
\begin{cases}
-\phi_t-a(t,x)\Delta\phi+b(t,x)\cdot D\phi=\psi\,,\\
\phi(T)=\phi_T\,,
\end{cases}
$$
it holds
$$
\int_{\R^d}\phi_T m(T,dx)+\int_0^T\int_{\R^d}\psi(t,x)m(t,dx)dt=\int_{\R^d}\phi(0,x)m_0(dx)\,.
$$
\end{defn}

Before starting the study of the Mean-Field Games system, we need to prove the following Proposition about some regularity estimates of $m$.

\begin{prop}\label{prop:fp}
Let $m$ be the unique solution in $\mathcal C([0,T];\mathcal P(\R^d))$ of \eqref{eq:fp}, in the sense of Definition \ref{def:FP}, where $a$ is uniformly elliptic, $a$, $b$ bounded in $L^\infty$, and $m_0\in\mathcal P(\R^d)$. Then $m$ satisfies, for a certain $p_0>1$ and $C>0$ both independent of $m_0$,
\begin{equation}\label{eq:stimem}
\begin{split}
\norm{m}_{L^p([0,T]\times\R^d)}+\sup\limits_{s\neq t}\frac{\dw(m(t),m(s))}{|t-s|^\miezz}+\sup\limits_{t\in[0,T]}\int_{\R^d}|x|\,m(t,dx)\le C\,,
\end{split}
\end{equation}
for all $p<p_0$.
\end{prop}
\begin{proof}
We start assuming $m_0$ smooth, and we consider the process satisfying
$$
\begin{cases}
dX_t=b(t,X_t)\,dt+\sigma(t,X_t)\,dB_t\,,\\
X_0=Z\,,
\end{cases}
$$
where $Z$ is a random variable of law $m_0$, $(B_t)_t$ is a standard Brownian motion and $\sigma\sigma^*=a$. Then, assuming without loss of generality that $t>s$, we have
\begin{equation*}
\begin{split}
\E[|X_t-X_s|]\le\, &\E\left[\int_s^t|b(r,X_r)|\,dr\right]+\E\left[\left|\int_s^t\sigma(r,X_r)\,dB_r\right|\right]\\
\le\, & C|t-s|+\E\left[\int_s^t|\sigma(r,X_r)|^2\,dr\right]^\miezz\le C|t-s|^\miezz\,.
\end{split}
\end{equation*}
Since $X_t$ has law $m(t)$ and $X_s$ has law $m(s)$, we obtain
$$
\sup\limits_{s\neq t}\frac{\dw(m(t),m(s))}{|t-s|^\miezz}\le C\,.
$$
To prove the $L^p$ bound for $m$, we take $\psi\in L^q$, with $\psi\ge0$ a.e. and for a certain $q$ which will be specified later. Then we consider a sequence $\psi^n\in\mathcal C^\infty$ with $\psi^n\to\psi$ in $L^q$, and we take $\phi^n$ as the solution in $[0,T]\times\R^d$ of
$$
\begin{cases}
-\phi^n_t-a(t,x)\Delta\phi^n+b(t,x)\cdot D\phi^n=\psi^n\,,\\
\phi^n(T)=0\,,
\end{cases}
$$
where $\psi\in L^q(\R^d)$  Then, \emph{Theorem IV.9.1} of \cite{lsu} tells us that
$$
\norm{\phi^n}_{W^{1,q}([0,T]\times\R^d)}\le C\norm{\psi^n}_{L^q}\le C\norm{\psi}_{L^q}\,,
$$
for a possibly different constant $C$ in the last inequality. Hence, for $q>n+1$, $\phi^n$ satisfies a H\"older estimate in space and time variables, uniformly in $n$. Multiplying the equation of $m$ by $\phi^n$ and integrating by parts, we obtain
$$
\int_0^T\int_{\R^d}\psi^n(t,x)m(t,x)\,dxdt=\int_{\R^d}\phi^n(0,x)m_0(dx)\le C\norm{\psi}_{L^q}\,,
$$
Applying Fatou's Lemma, this implies that
$$
\norm{m}_{L^p(\R^d)}\le C\,,
$$
for all $p<p_0$, where $p_0=1+\frac 1n$ is the conjugate exponent of $n+1$.

Finally, we consider $\phi$ as the solution in $[0,t]\times\R^d$ of
$$
\begin{cases}
-\phi_t-a(s,x)\Delta\phi+b(s,x)\cdot D\phi=0\,,\\
\phi(t)=|x|\,.
\end{cases}
$$
From standard regularity results (see e.g. \cite{PorrettaPriola}), we know that
$$
\sup\limits_{s\in[0,t]}\norminf{\phi(s,\cdot)}\le C(1+|x|)\,.
$$
Hence, multiplying the equation of $m$ by $\phi$ and integrating by parts, we obtain
$$
\int_{\R^d}|x| m(t,x)\,dxdt=\int_{\R^d}\phi(0,x)\,m_0(dx)\le C\left(1+\int_{\R^d}m_0(dx)\right)\le C,,
$$
since $m_0\in\mathcal P(\R^d)$.

Since these estimates do not depend on the smoothness of $m_0$, with a standard approximation we can obtain \eqref{eq:stimem} $\forall\,m_0\in\mathcal P(\R^d)$, which concludes the Proposition.
\end{proof}

Now we are ready to study the full Mean-Field Games system.

We handle the case where our couplings $F$ and $G$ are \emph{non-local}.
\begin{hp}
$F:[0,T]\times\R^d\times\mathcal P(\R^d)\to\R$ is a $\mathcal C^{1+\alpha}$ function in the space variable and $\mathcal C^{\frac\alpha 2}$ in the time variable, for a certain $\alpha\in(0,1)$, satisfying
$$
\sup\limits_{m\in\mathcal P(\R^d)}\norm{F(\cdot,\cdot, m)}_{\frac\alpha 2,1+\alpha}\le C\,.
$$
Futhermore, $G$ is a $\mathcal C^{3+\alpha}$ function in the space variable, and $F$ and $G$ satisfy a Lipschitz estimate with respect to the measure:
\begin{equation*}
\begin{split}
\norm{F(t,\cdot,m_1)-F(t,\cdot,m_2)}_{1+\alpha}\le C\dw(m_1,m_2)\,, \qquad& \forall m_1,\,m_2\in\mathcal P(\R^d)\,,\\
\norm{G(\cdot,m_1)-G(\cdot,m_2)}_{3+\alpha}\le C\dw(m_1,m_2)\,, \qquad& \forall m_1,\,m_2\in\mathcal P(\R^d)\,.
\end{split}
\end{equation*}
\end{hp}

Observe that the Lipschitz character of $F$ and $G$ with respect to the measure variable is not a restrictive hypothesis, since in many applications the functions $F$ and $G$ have a linear or sublinear growth, both in the space and in the measure variable. See e.g. \cite{FestaRicciardi,GozziMasieroRosestolato,RicciardiRosestolato}.

We begin with the existence part.
\begin{thm}
Suppose the hypotheses of Proposition \ref{prop:prop} are satisfied, and suppose $m_0\in\mathcal P(\R^d)$. Then there exists at least one solution $(u,m)\in\mathcal C^{\frac{3+\alpha}2,3+\alpha}\times\mathcal C([0,T];\mathcal P(\R^d))$ for the Mean-Field Games system \eqref{eq:mfg}.

Furthermore, if one of these conditions is satisfied:
\begin{itemize}
\item [(i)] $H^1$ and $H^2$ strictly concave with respect to the last variable and $F$ and $G$ non-decreasing with respect to $m$:
\begin{equation}\label{hp:monoton}
\begin{split}
\int_{\R^d}\big(F(t,x,m_1)-F(t,x,m_2)\big)(m_1(dx)-m_2(dx))\ge0\,,\\
\int_{\R^d}\big(G(x,m_1)-G(x,m_2)\big)(m_1(dx)-m_2(dx))\ge0\,.
\end{split}
\end{equation}
\item [(ii)] $H^1$ and $H^2$ concave with respect to the last variable and $F$ and $G$ strictly increasing with respect to $m$, i.e. $F$ and $G$ satisfy \eqref{hp:monoton} and in addition:
\begin{equation*}
\begin{split}
&\int_{\R^d}\big(F(t,x,m_1)-F(t,x,m_2)\big)(m_1(dx)-m_2(dx))=0\implies F(t,x,m_1)=F(t,x,m_2)\,,\\
&\int_{\R^d}\big(G(x,m_1)-G(x,m_2)\big)(m_1(dx)-m_2(dx))=0\implies G(x,m_1)=G(x,m_2)\,,
\end{split}
\end{equation*}
\end{itemize}
then the solution is unique.
\end{thm}

\begin{proof}
The existence relies on Schauder fixed point Theorem. Similar ideas were developed in \cite{CDLL,RicciardiDirichlet,RicciardiNeumann}.

We consider the following metric space:
$$
X=\left\{\gamma\in\mathcal C([0,T];\mathcal P(\R^d))\,:\,\dw(\gamma(t),\gamma(s))\le C|t-s|^\miezz\right\}\,,
$$
where $C$ will be defined later. It is immediate to note that $X$ is a convex closed space.

We want to apply the Schauder fixed point Theorem. First, we define a suitable functional $\Phi$.

For $\gamma\in X$ we consider $u$ as the solution of the Hamilton-Jacobi equation
\begin{equation}\label{eq:phihjb}
\left\{
\begin{array}{ll}
\displaystyle u_t+H^2(t,x,\Delta u)+H^1(t,x,\nabla u)+F(t,x,\gamma(t))=0\,, & (t,x)\in[0,T]\times\R^d\,,\\
\displaystyle u(T,x)=G(x,\gamma(T))\,, & x\in\R^d\,,
\end{array}\right.
\end{equation}
and then we define $\Phi(\gamma)=m$, where $m$ is the solution of
\begin{equation}\label{eq:phifp}
\left\{
\begin{array}{ll}
\displaystyle m_t-\Delta(mH^2_q(t,x,\Delta u))+\mathrm{div}(mH^1_p(t,x,\nabla u))=0\,, & (t,x)\in(0,T)\times\R^d\,,\\
\displaystyle m(0,x)=m_0(x)\,, & x\in\R^d\,.
\end{array}\right.
\end{equation}
Thanks to the regularity results for $u$, we know that $m$ is well-defined in the space $\mathcal C([0,T];\mathcal P(\R^d))$.

Proposition \ref{prop:fp} implies that, if we choose wisely $C$ in the definition of $X$, we have $m\in X$ and, in addition to that, $\norm{m}_{L^p}\le C$, with $C$ not depending on $\gamma$.

In order to apply Schauder's Theorem we need to show the continuity of $\Phi$ and that $\Phi(X)$ is a relatively compact set.

We start by showing that $\Phi(X)$ is relatively compact.

Let $\{\gamma_n\}_n\subset X$, and let $u_n$ and $m_n$ be the solutions of \eqref{eq:phihjb} and \eqref{eq:phifp} related to $\gamma_n$. From \eqref{eq:5.1} we know that
$$
\norm{u_n}_{1+\frac\alpha 2,3+\alpha}\le C\,,
$$
where $C$ does not depend on $n$. Hence, Ascoli-Arzelà  and Banach-Alaoglu Theorems imply that $\exists\{u_{n_k}\}_k$, $u\in\mathcal C^{1+\frac\alpha2,3+\alpha}$ such that $u_{n_k}\to u$ in $\mathcal C^{1,3}_{loc}$ and all the derivatives converge pointwise.

For simplicity, we call the subsequence $u_{n_k}$ as $u_n$, forgetting the dependence on $k$.

To prove the convergence of $\{m_n\}_n$ (up to subsequences), we readapt the ideas from \cite{FestaRicciardi}. Consider the processes $X_s^n$ and $X_s^k$ satisfying
\begin{equation*}
\begin{split}
dX_s^n&=H_p^1(s,X_s^n,\nabla u_n(s,X_s^n))ds+\sqrt{2H^2_q(s,X_s^n,\Delta u_n(S,X_s^n))}dB_s\,,\\
dX_s^k&=H_p^1(s,X_s^k,\nabla u_k(s,X_s^k))ds+\sqrt{2H^2_q(s,X_s^k,\Delta u_k(S,X_s^k))}dB_s\,,
\end{split}
\end{equation*}
with $X_0^n=X_0^k=Z$, a process with density $m_0$. Then we have
\begin{equation}\label{eq:6.7}
\begin{split}
\E\left[|X_t^n-X_t^k|^2\right]&\le\E\left[\int_0^t|H_p^1(s,X_s^n,\nabla u_n(s,X_s^n))-H_p^1(s,X_s^k,\nabla u_k(s,X_s^k))|^2\,ds\right]\\
&+\E\left[\left|\int_0^t\left(\sqrt{2H^2_q(s,X_s^n,\Delta u_n(s,X_s^n))}-\sqrt{2H^2_q(s,X_s^k,\Delta u_k(s,X_s^k))}\right)dB_s\right|^2\right]\,.
\end{split}
\end{equation}
We analyze each term. The last term is equal to
\begin{equation*}
\begin{split}
&\E\left[\int_0^t\left(\sqrt{2H^2_q(s,X_s^n,\Delta u_n(s,X_s^n))}-\sqrt{2H^2_q(s,X_s^k,\Delta u_k(s,X_s^k))}\right)^2ds\right]\\
\le\,&\E\left[\int_0^t\left(\sqrt{2H^2_q(s,X_s^n,\Delta u_n(s,X_s^n))}-\sqrt{2H^2_q(s,X_s^n,\Delta u_k(s,X_s^n))}\right)^2ds\right]\\
+\,&\E\left[\int_0^t\left(\sqrt{2H^2_q(s,X_s^n,\Delta u_k(s,X_s^n))}-\sqrt{2H^2_q(s,X_s^k,\Delta u_k(s,X_s^k))}\right)^2ds\right]\\
\le\,&C\int_0^t\int_{\R^d}|\Delta(u_n-u_k)|^2m_n(s,x)\,dxds+C\int_0^t\E\left[|X_s^n-X_s^k|^2\right]\,ds
\end{split}
\end{equation*}
where we use the Lipschitz bounds of $H^2_q$ and $\Delta u_k$.

In order to handle the first integral in the right-hand side, we consider a bounded domain $E\subset\R^d$ and we write
\begin{equation*}
\begin{split}
&\int_0^t\int_{\R^d}|\Delta(u_n-u_k)|^2 m_n(s,x)\,dxds\\
=&\int_0^t\int_E|\Delta(u_n-u_k)|^2m_n(s,x)\,dxds+\int_0^t\int_{E^c}|\Delta(u_n-u_k)|^2m_n(s,x)\,dxds\,.
\end{split}
\end{equation*}
The first integral goes to $0$ for any $E$ bounded, thanks to the uniform convergence of $\Delta u_n$ in bounded sets and the $L^p$ bound of $m_n$. For the second integral, we note that $|\Delta(u_n-u_k)|^2$ is uniformly bounded and that $m_n$ has a finite first order moment thanks to \eqref{eq:stimem}. Hence, $\forall\,\eps>0$ we can choose $E$ sufficiently large such that
$$
\int_0^t\int_{E^c}|\Delta(u_n-u_k)|^2m_n(s,x)\,dxds\le\eps\,.
$$
For the arbitrariness of $\eps$, we obtain
\begin{equation*}
\begin{split}
&\E\left[\left|\int_0^t\left(\sqrt{2H^2_q(s,X_s^n,\Delta u_n(s,X_s^n))}-\sqrt{2H^2_q(s,X_s^k,\Delta u_k(s,X_s^k))}\right)dB_s\right|^2\right]\\
&\le\omega(n,k)+C\int_0^t\E\left[|X_s^n-X_s^k|^2\right]\,,
\end{split}
\end{equation*}
where $\omega(n,k)$ is a quantity converging to $0$ when $n,k\to+\infty$.

Similar estimates are made in order to bound the first term of \eqref{eq:6.7}. Hence we get
$$
\E\left[|X_t^n-X_t^k|^2\right]\le\omega(n,k)+C\int_0^t\E\left[|X_s^n-X_s^k|^2\right]\,,
$$
and so, applying Gronwall's inequality,
$$
\E[|X_t^n-X_t^k|]\le\E[|X_t^n-X_t^k|^2]\le C\omega(n,k)\,,
$$
which immediately implies, up to changing $\omega(n,k)$,
$$
\sup\limits_{t\in[0,T]}\dw(m_n(t),m_k(t))\le\omega(n,k)\,.
$$
Hence, we have proved that $\{m_n\}_n$ is a Cauchy sequence in $X$. Then, up to subsequences, $\exists\,m$ such that $m_n\to m$. This proves the relatively compactness of $\Phi(X)$.

The continuity is an easy consequence of the compactness. We consider $\gamma_n\to\gamma$. From the compactness, there exists subsequences $\{u_{n_k}\}_k$, $\{m_{n_k}\}_k$ converging to some $u$, $m$. Passing to the limit in the formulations of $u_{n_k}$ and $m_{n_k}$, we obtain that $u$ and $m$ are the (unique) solutions of \eqref{eq:phihjb} and \eqref{eq:phifp} related to $\gamma$ and $u$.

Then, for each converging subsequence $\{m_{n_h}\}_h$, we must have $m_{n_h}\to m$. This means that the whole sequence $\Phi(\gamma_n)=m_n\to m=\Phi(\gamma)$. This proves the continuity of $\Phi$.

So by Schauder's Theorem and obtain a classical solution of the MFG system \eqref{eq:mfg}.\\

To prove the uniqueness, let $(u_1,m_1)$, $(u_2,m_2)$ be two solutions of \eqref{eq:mfg}. We want to estimate in two different ways the quantity
$$
\int_0^T\int_{\R^d}\big[(u_1-u_2)(m_1-m_2)\big]_t\,dxdt\,.
$$
First, computing directly the time integral we obtain
\begin{equation}\label{eq:6.8}
\begin{split}
\int_{\R^d}\big[G(x,m_1(T))-G(x,m_2(T))\big](m_1(T)-m_2(T))\,dx\ge0\,.
\end{split}
\end{equation}
On the other hand, if we compute the derivative and use the weak formulation of $(u_1,m_1)$ and $(u_2,m_2)$, we obtain
\begin{equation*}
\begin{split}
-\int_0^T\int_{\R^d}\big[F(t,x,m_1(t))-F(t,x,m_2(t))\big](m_1(t,dx)-m_2(t,dx))\,dt\,\\
+\int_0^T\int_{\R^d}\big(H^2(t,x,\Delta u_2)-H^2(t,x,\Delta u_1)-H^2_q(t,x,\Delta u_1)\Delta(u_2-u_1)\big)\,m_1(t,dx)dt\,\\
+\int_0^T\int_{\R^d}\big(H^2(t,x,\Delta u_1)-H^2(t,x,\Delta u_2)-H^2_q(t,x,\Delta u_2)\Delta(u_1-u_2)\big)\,m_2(t,dx)dt\,\\
+\int_0^T\int_{\R^d}\big(H^1(t,x, Du_2)-H^1(t,x, Du_1)-H^1_p(t,x, Du_1)D(u_2-u_1)\big)\,m_1(t,dx)dt\,\\
+\int_0^T\int_{\R^d}\big(H^1(t,x, Du_1)-H^1(t,x, Du_2)-H^1_p(t,x, Du_2)D(u_1-u_2)\big)\,m_2(t,dx)dt.
\end{split}
\end{equation*}
Since $H^2$ and $H^1$ are concave functions, all the above integrals are non-positive.

Then, combining this result with \eqref{eq:6.8}, we obtain
$$
\int_0^T\int_{\R^d}\big[(u_1-u_2)(m_1-m_2)]_t=0\,,
$$
which means
\begin{equation*}
\begin{split}
\int_{\R^d}\big[G(x,m_1(T))-G(x,m_2(T))\big](m_1(T)-m_2(T))\,dx=0\,,\\
\int_0^T\int_{\R^d}\big[F(t,x,m_1(t))-F(t,x,m_2(t))\big](m_1(t,dx)-m_2(t,dx))\,dt=0\,,\\
\int_0^T\int_{\R^d}\big(H^2(t,x,\Delta u_i)-H^2(t,x,\Delta u_j)-H^2_q(t,x,\Delta u_j)\Delta(u_i-u_j)\big)\,m_j(t,dx)dt=0\,,\\
\int_0^T\int_{\R^d}\big(H^1(t,x, Du_i)-H^1(t,x, Du_j)-H^1_p(t,x, Du_j)D(u_i-u_j)\big)\,m_j(t,dx)dt=0\,,
\end{split}
\end{equation*}
for $(i,j)=(1,2)$ or $(2,1)$.

This allows us to conclude: actually, thanks to the hypotheses, there are two cases:
\begin{itemize}
\item [(i)] if $F$ and $G$ are strictly increasing, then
$$
F(t,x,m_1(t))=F(t,x,m_2(t))\,,\qquad G(x,m_1(T))=G(x,m_2(T))\,.
$$
Hence, $u_1$ and $u_2$ solve the same HJB equation, which implies $u_1=u_2$. Coming back to the FP equation, we have that $\Delta u_1=\Delta u_2$ and $Du_1=Du_2$. So $m_1$ and $m_2$ solve the same FP equation, which implies $m_1=m_2$;
\item [(ii)] if $H^1$ and $H^2$ are strictly concave, we argue in a similar way. Actually, we have, for all $t\in[0,T]$, for $x\in\,supp(m_1(t))\cup\,supp(m_2(t))$,
$$
H_p^1(t,x,Du_1)=H_p^1(t,x,Du_2)\,,\qquad H^2_q(t,x,\Delta u_1)=H^2_q(t,x,\Delta u_2)\,,
$$
which implies
$$
m_1H_p^1(t,x,Du_1)=m_1 H_p^1(t,x,Du_2)\,,\qquad m_1 H^2_q(t,x,\Delta u_1)=m_1 H^2_q(t,x,\Delta u_2)\,.
$$
Hence, $m_1$ and $m_2$ solve the same FP equation, which implies $m_1=m_2$. Coming back to the HJB equation, we have that $F(t,x,m_1(t))=F(t,x,m_2(t))$ and $G(x,m_1(T))=G(x,m_2(T))$. So $u_1$ and $u_2$ solve the same HJB equation, which implies $u_1=u_2$.
\end{itemize}
Hence, the proof of uniqueness is completed.
\end{proof}

\section*{Acknowledgments} The authors wish to sincerely thank P. Cardaliaguet, F. Da Lio, A. Porretta and M. Soner for the help and the support during the preparation of this article. We wish also to thank A. Neufeld and F. Delarue for the inspiring suggestions he gave to us.

\end{document}